\documentclass[reqno,11pt,a4paper]{amsart}

\usepackage{amssymb}
\usepackage[margin=2.7cm]{geometry}
\usepackage{tikz}
\usepackage{mathtools}
\usepackage{enumitem}
\usepackage{subfig}
\usetikzlibrary{shapes.misc}
\usepackage{hyperref}
\hypersetup{
    colorlinks,
    linkcolor={red!50!black},
    citecolor={blue!50!black},
    urlcolor={blue!80!black}
}
\usepackage[ruled,vlined]{algorithm2e}
\usepackage{booktabs}
\newcommand\headercell[1]{\smash[b]{\begin{tabular}[t]{@{}c@{}} #1 \end{tabular}}}

\graphicspath{{./Plots/}}
\tikzset{cross/.style={cross out, draw=black, minimum size=2*(#1-\pgflinewidth), inner sep=0pt, outer sep=0pt}, cross/.default={1pt}}

\newcommand{\ZZ}{\mathbb{Z}}
\newcommand{\RR}{\mathbb{R}}
\newcommand{\QQ}{\mathbb{Q}}
\newcommand{\PP}{\mathbb{P}}
\newcommand{\cT}{\mathcal{T}}
\newcommand{\cS}{\mathcal{S}}
\newcommand{\cP}{\mathcal{P}}
\newcommand{\cW}{\mathcal{W}}
\newcommand{\cA}{\mathcal{A}}

\DeclareMathOperator{\conv}{conv}
\DeclareMathOperator{\width}{width}
\DeclareMathOperator{\mwidth}{mwidth}

\DeclareMathOperator{\Hom}{Hom}
\DeclareMathOperator{\NF}{NF}

\newtheorem{theorem}{Theorem}[section]
\newtheorem{proposition}[theorem]{Proposition}
\newtheorem{lemma}[theorem]{Lemma}
\theoremstyle{definition}
\newtheorem{definition}[theorem]{Definition}
\newtheorem{conjecture}[theorem]{Conjecture}
\newtheorem{corollary}[theorem]{Corollary}

\begin{document}

\author[G.\ Hamm]{Girtrude~Hamm}
\address{School of Mathematical Sciences\\University of Nottingham\\Nottingham\\NG7 2RD\\UK}
\email{girtrude.hamm@nottingham.ac.uk}


\title{Classification of width 1 lattice tetrahedra by their multi-width}

\begin{abstract}
We introduce the multi-width of a lattice polytope and use this to classify and count all lattice tetrahedra with multi-width \((1,w_2,w_3)\).
The approach used in this classification can be extended into a computer algorithm to classify lattice tetrahedra of any given multi-width.
We use this to classify tetrahedra with multi-width \((2,w_2,w_3)\) for small \(w_2\) and \(w_3\) and make conjectures about the function counting lattice tetrahedra of any multi-width.
\end{abstract}
\maketitle

\section{Introduction}
\label{sec:intro}

A \emph{lattice polytope} \(P \subseteq \RR^d\) is the convex hull of finitely many lattice points, that is, points in \(\ZZ^d\).
We consider lattice polytopes as being defined only up to affine unimodular equivalence.
Two lattice polytopes are said to be \emph{(affine) equivalent} if one can be mapped to the other by a change of basis of \(\ZZ^d\) followed by an integral translation.
A \emph{lattice simplex} is the convex hull of affinely independent lattice points.
For example, in dimensions 2 and 3 these are triangles and tetrahedra respectively.

Lattice simplices are recurring objects of study with multiple applications.
Via toric geometry they are relevant to algebraic geometry and are closely related to toric \(\QQ\)-factorial singularities.
The toric Fano three-folds with at most terminal singularities were classified by finding all the three-dimensional lattice polytopes whose only lattice points were the origin and their vertices \cite{AKClassTet}.
A key step towards this was classifying all such tetrahedra.
Simplices whose only lattice points are their vertices can give terminal quotient singularities by placing one vertex at the origin and considering the cone they generate.
These are called \emph{empty simplices} and were classified in dimension 3 and 4 in \cite{empty_dim_3_simpl} and \cite{empty_dim_4_simpl} respectively.
There are also applications of lattice simplices in mixed-integer and integer optimisation, see for example \cite{MaxLatFreePoly} and \cite{AverkovNill}.

An important affine invariant of a polytope is its width.
Recall that for a lattice polytope \(P \subseteq \RR^d\) and a primitive dual vector \(u \in (\ZZ^d)^*\) the width of \(P\) with respect to \(u\), written \(\width_u(P)\), is the length of the interval obtained by projecting \(P\) along the hyperplane with normal vector \(u\); that is, \(\width_u(P) \coloneqq \max_{x \in P}\{u \cdot x\}-\min_{x \in P}\{u \cdot x\}\).
Then the \emph{(first) width} of \(P\), written \(\width^1(P)\), is the minimum width along all non-zero dual vectors \(u\), i.e. \(\width^1(P) \coloneqq \min_{u \in (\ZZ^d)^*\setminus \{0\}} \{\width_u(P)\}\).
Width plays a role in the proofs of both \cite{empty_dim_4_simpl} and \cite{MaxLatFreePoly} mentioned above which motivates seeking an understanding of the simplices of a given width.
However, in dimension at least 2, there are infinitely many simplices of a given width.
We would like to record enough information about the widths of a polytope so that there are only finitely many polytopes satisfying these conditions.
To do this we consider the width of a lattice polytope in multiple directions.
For a linearly independent collection of dual vectors \(u_1, \dots, u_d \in (\ZZ^d)^*\) we can consider the tuple whose entries are \(\width_{u_i}(P)\).
By applying lexicographical order to \(\ZZ^d_{\geq 0}\) we find the minimum such tuple.
We call this the multi-width of \(P\) written \(\mwidth(P)\) and the \(i\)-th entry of this tuple is the \(i\)-th width of \(P\) written \(\width^i(P)\).
Since \(\width^i(P)\) is always greater than or equal to \(\width^{i-1}(P)\) from now on, unless otherwise specified, let \(w_1\), \(w_2\) and \(w_3\) be integers satisfying \(0 < w_1 \leq w_2 \leq w_3\).

This author completely classified lattice triangles by their multi-widths in \cite{triangles}.
The result is surprisingly simple and it produces a normal form for triangles from which both their width and automorphism groups can be easily read.
Additionally, it shows that the sequence counting lattice triangles with second width at most \(w_2\) has generating function equal to the Hilbert series of a degree 8 hypersurface in \(\PP(1,1,1,2,2,2)\).
Here we investigate how much of this can be extended to the three dimensions.
Ideally, we would describe the finite sets \(\cT_{w_1,w_2,w_3}\) defined as follows.
\begin{definition}\label{def:tet_up_to_aff}
For integers \(0 < w_1 \leq w_2 \leq w_3\) the set of tetrahedra with multi-width \((w_1,w_2,w_3)\) up to affine equivalence (denoted by \(\sim\)) is
\[
\cT_{w_1,w_2,w_3} \coloneqq \{T = \conv(v_1,v_2,v_3,v_4) : v_i \in \ZZ^3, \mwidth(T) = (w_1,w_2,w_3)\}/\sim.
\]
\end{definition}
Theorem~\ref{thm:main} achieves this when the first width is 1 by establishing a bijection between \(\cT_{1,w_2,w_3}\) and a set of tetrahedra \(\cS_{1,w_2,w_3}\).
To define \(\cS_{1,w_2,w_3}\) we must first recall the classification of lattice triangles by their multi-width, then define the four types of tetrahedron which our new classification will include.

\begin{definition}\label{def:tri_class}
Let \(\cS_{w_1,w_2}\) be the set of lattice triangles
\begin{itemize}
    \item \(\conv((0,0),(w_1,y_1),(0,w_2))\) where \(0 \leq y_1 \leq (w_2-y_1 \mod w_1)\),
    \item \(\conv((0,0),(w_1,y_1),(x_2,w_2))\) where \(0<x_2\leq \frac{w_1}{2}\) and \(0 \leq y_1 \leq w_1-x_2\) \\(and \(y_1 \geq x_2\) if \(w_1=w_2\))
    \item and if \(w_1<w_2\), \(\conv((0,y_0),(w_1,0),(x_2,w_2))\) where \(1 < x_2 < \frac{w_1}{2}\) and \(0 < y_0 < x_2\).
\end{itemize}
\end{definition}
As shown in \cite[Theorem 1.2]{triangles} the map taking a lattice triangle to its affine equivalence class is a bijection from \(\cS_{w_1,w_2}\) to the set of lattice triangles with multi-width \((w_1,w_2)\) up to affine equivalence.

We now define \(\cS_{1,w_2,w_3}\), the main classification result of the paper:
\begin{definition}\label{def:tet_class}
The four types of tetrahedron which will appear in \(\cS_{1,w_2,w_3}\) are
\begin{enumerate}
    \item\label{item:tet_0001} \(\conv(\{0\} \times t, (1,0,0))\) where \(t \in \cS_{w_2,w_3}\),
    \item\label{item:tet_0011_0} \(\conv((0,0,0),(0,w_2,z_1),(1,0,0),(1,0,w_3))\) where \(0 \leq z_1 \leq \frac{w_2}{2} \),
    \item\label{item:tet_0011_1} \(\conv((0,0,0),(0,w_2,z_1),(1,0,w_3),(1,y_1,0))\) where \(0 < y_1 \leq w_2\) and \(w_3-w_2 \leq z_1 \leq w_3\) and
    \item\label{item:tet_0011_2} \(\conv((0,0,0),(0,w_2,w_3),(1,0,w_3),(1,y_1,z_1))\) where \(0 < z_1 < y_1 < w_2\).
\end{enumerate}
For examples of these see Figure~\ref{fig:tet_eg}.

If \(w_3>w_2>1\) let \(\cS_{1,w_2,w_3}\) be the set containing all tetrahedra of type~\ref{item:tet_0001}-\ref{item:tet_0011_2}.

If \(w_3=w_2>1\) then \(\cS_{1,w_2,w_2}\) is the set of all type~\ref{item:tet_0001} and \ref{item:tet_0011_0} tetrahedra as well as type~\ref{item:tet_0011_1} tetrahedra satisfying \(y_1 \leq z_1\) and type~\ref{item:tet_0011_2} tetrahedra satisfying \(z_1 \leq w_2-y_1\). 

If \(w_3>w_2=1\) then
\begin{align*}
\cS_{1,1,w_3} \coloneqq \{& \conv((0,0,0),(0,1,0),(0,0,w_3),(1,0,0)),\\& \conv((0,0,0),(0,1,w_3-1),(1,0,w_3),(1,1,0)), \\
& \conv((0,0,0),(0,1,w_3),(1,0,w_3),(1,1,0))\}.
\end{align*}

If \(w_3=w_2=1\) then
\begin{align*}
\cS_{1,1,1} \coloneqq \{&\conv((0,0,0),(0,1,0),(0,0,1),(1,0,0)),\\ & \conv((0,0,0),(0,1,1),(1,0,1),(1,1,0))\}.
\end{align*}
These last two cases have elements of type~\ref{item:tet_0001} and \ref{item:tet_0011_1}.
\end{definition}

\begin{figure}[ht]
\centering
\subfloat[Type~\ref{item:tet_0001}]{
\begin{tikzpicture}[x=0.4cm,y=0.4cm]
\draw[very thick,fill=gray!30] (0,0) -- (6,3) -- (2,7) -- cycle;
\draw[very thick] (1,1) -- (0,0);
\draw[very thick] (1,1) -- (6,3);
\draw[very thick] (1,1) -- (2,7);

\draw[dashed] (0,0) -- (6,0) -- (6,7) -- (0,7) -- cycle;
\draw[dashed] (1,1) -- (7,1) -- (7,8) -- (1,8) -- cycle;
\draw[dashed] (0,0) -- (1,1);
\draw[dashed] (6,0) -- (7,1);
\draw[dashed] (6,7) -- (7,8);
\draw[dashed] (0,7) -- (1,8);

\node[draw,circle,inner sep=2pt,fill=white] at (0,0) {};
\node[draw,circle,inner sep=2pt,fill=white] at (6,3) {};
\node[draw,circle,inner sep=2pt,fill=white] at (2,7) {};
\node[draw,circle,inner sep=2pt,fill] at (1,1) {};
\node[] at (0-1.8,0) {(0,0,0)};
\node[] at (8.3,3) {(0,6,3)};
\node[] at (2,8.5) {(0,2,7)};
\node[] at (1-2.4,1) {(1,0,0)};
\end{tikzpicture}}
\subfloat[Type~\ref{item:tet_0011_0}]{
\begin{tikzpicture}[x=0.4cm,y=0.4cm]
\draw[dashed] (0,0) -- (6,0) -- (6,7) -- (0,7) -- cycle;
\draw[very thick,fill=gray!30] (0,0) -- (6,3) -- (1,8) -- cycle;
\draw[very thick] (1,1) -- (0,0);
\draw[very thick] (1,1) -- (6,3);
\draw[very thick] (1,1) -- (1,8);

\draw[dashed] (1,1) -- (7,1) -- (7,8) -- (1,8) -- cycle;
\draw[dashed] (0,0) -- (1,1);
\draw[dashed] (6,0) -- (7,1);
\draw[dashed] (6,7) -- (7,8);
\draw[dashed] (0,7) -- (1,8);

\node[draw,circle,inner sep=2pt,fill] at (0,0) {};
\node[draw,circle,inner sep=2pt,fill=white] at (6,3) {};
\node[draw,circle,inner sep=2pt,fill] at (1,1) {};
\node[draw,circle,inner sep=2pt,fill] at (1,8) {};
\node[] at (0-1.8,0) {(0,0,0)};
\node[] at (8.3,3) {(0,6,3)};
\node[] at (2,8.5) {(1,0,7)};
\node[] at (1-2.4,1) {(1,0,0)};
\end{tikzpicture}}\\
\subfloat[Type~\ref{item:tet_0011_1}]{
\begin{tikzpicture}[x=0.4cm,y=0.4cm]
\draw[dashed] (0,0) -- (6,0) -- (6,7) -- (0,7) -- cycle;
\draw[very thick,fill=gray!30] (0,0) -- (4,1) -- (6,5) -- (1,8) -- cycle;
\draw[very thick] (4,1) -- (1,8);
\draw[dotted] (0,0) -- (6,5);

\draw[dashed] (1,1) -- (7,1) -- (7,8) -- (1,8) -- cycle;
\draw[dashed] (0,0) -- (1,1);
\draw[dashed] (6,0) -- (7,1);
\draw[dashed] (6,7) -- (7,8);
\draw[dashed] (0,7) -- (1,8);

\node[draw,circle,inner sep=2pt,fill] at (0,0) {};
\node[draw,circle,inner sep=2pt,fill=white] at (4,1) {};
\node[draw,circle,inner sep=2pt,fill=white] at (6,5) {};
\node[draw,circle,inner sep=2pt,fill] at (1,8) {};
\node[] at (0-1.8,0) {(0,0,0)};
\node[] at (8.3,5) {(0,6,5)};
\node[] at (2,8.5) {(1,0,7)};
\node[] at (5,-0.5) {(1,3,0)};
\end{tikzpicture}}
\subfloat[Type~\ref{item:tet_0011_2}]{
\begin{tikzpicture}[x=0.4cm,y=0.4cm]
\draw[dashed] (0,0) -- (6,0) -- (6,7) -- (0,7) -- cycle;
\draw[very thick,fill=gray!30] (0,0) -- (5,3) -- (6,7) -- (1,8) -- cycle;
\draw[very thick] (5,3) -- (1,8);
\draw[dotted] (0,0) -- (6,7);

\draw[dashed] (1,1) -- (7,1) -- (7,8) -- (1,8) -- cycle;
\draw[dashed] (0,0) -- (1,1);
\draw[dashed] (6,0) -- (7,1);
\draw[dashed] (6,7) -- (7,8);
\draw[dashed] (0,7) -- (1,8);
\node[draw,circle,inner sep=2pt,fill] at (0,0) {};
\node[draw,circle,inner sep=2pt,fill=white] at (5,3) {};
\node[draw,circle,inner sep=2pt,fill] at (6,7) {};
\node[draw,circle,inner sep=2pt,fill] at (1,8) {};
\node[] at (0-1.8,0) {(0,0,0)};
\node[] at (8.3,7) {(0,6,7)};
\node[] at (2,8.5) {(1,0,7)};
\node[] at (8.3,3) {(1,4,2)};
\end{tikzpicture}}
\caption{Examples of tetrahedra of type~\ref{item:tet_0001}-\ref{item:tet_0011_2} when \(w_2=6\) and \(w_3=7\).
Black vertices are fixed for a given type while white vertices are variable.}
\label{fig:tet_eg}
\end{figure}

We can now state the main theorem of this paper.
\newpage

\begin{theorem}\label{thm:main}
There is a bijection from \(\cS_{1,w_2,w_3}\) to \(\cT_{1,w_2,w_3}\) given by the map taking a tetrahedron to its affine equivalence class.
In particular, 
\begin{itemize}
    \item when \(w_3>w_2>1\) the cardinality of \(\cT_{1,w_2,w_3}\) is
    \begin{itemize}
        \item \(2w_2^2+4\) if \(w_2\) and \(w_3\) even
        \item \(2w_2^2+3\) if \(w_2\) even and \(w_3\) odd
        \item \(2w_2^2+2\) if \(w_2\) odd
    \end{itemize}
    \item when \(w_2>1\) the cardinality of \(\cT_{1,w_2,w_2}\) is
    \begin{itemize}
        \item \(w_2^2+w_2+2\) if \(w_2\) even
        \item \(w_2^2+w_2+1\) if \(w_2\) odd
    \end{itemize}
    \item when \(w_3>1\) the cardinality of \(\cT_{1,1,w_3}\) is \(3\)
    \item and  the cardinality of \(\cT_{1,1,1}\) is \(2\).
\end{itemize}
\end{theorem}

\begin{table}[ht]
    \centering
    \begin{tabular}{@{} ccccccccccccc @{}}
    \headercell{} & \multicolumn{12}{c@{}}{\(w_3\)}\\
    \cmidrule(l){2-13}
    \(w_2\) & 1 & 2 & 3 & 4 & 5 & 6 & 7 & 8 & 9 & 10 & 11 & 12\\
    \midrule
    1 & 2 & 3 & 3 & 3 & 3 & 3 & 3 & 3 & 3 & 3 & 3 & 3\\
    2 & 0 & 8 & 11 & 12 & 11 & 12 & 11 & 12 & 11 & 12 & 11 & 12\\
    3 & 0 & 0 & 13 & 20 & 20 & 20 & 20 & 20 & 20 & 20 & 20 & 20\\
    4 & 0 & 0 & 0 & 22 & 35 & 36 & 35 & 36 & 35 & 36 & 35 & 36\\
    5 & 0 & 0 & 0 & 0 & 31 & 52 & 52 & 52 & 52 & 52 & 52 & 52 \\
    6 & 0 & 0 & 0 & 0 & 0 & 44 & 75 & 76 & 75 & 76 & 75 & 76
\end{tabular}
    \caption{The number of lattice tetrahedra with multi-width \((1,w_2,w_3)\) up to affine equivalence for small \(w_2\) and \(w_3\).}
    \label{tab:width_1_tet}
\end{table}

Table~\ref{tab:width_1_tet} gives the cardinality of \(\cT_{1,w_2,w_3}\) when \(w_2 \leq 6\) and \(w_3 \leq 12\) described by Theorem~\ref{thm:main}.
The idea of the proof is to first show that a tetrahedron with multi-width \((1,w_2,w_3)\) is equivalent to a subset of \([0,1] \times [0,w_2] \times [0,w_3]\).
We then successively classify the possible \(x\)-, \(y\)- and \(z\)-coordinates of the vertices of the tetrahedra.
At each step we remove cases which have too small a width in some direction or are equivalent to other cases.
Two distinct vertices of a tetrahedron may have the same \(x\)- and \(y\)-coordinates so we need to consider multi-sets of lattice points.
In an abuse of notation we will write \(\{v_1,\dots,v_n\}\) for the \(n\)-point multi-set containing lattice points \(v_i \in \ZZ^d\) even when the \(v_i\) are not distinct.
We extend the notion of widths to these sets by saying the width of a set is the width of its convex hull.

We can completely classify the four-point sets in \(\ZZ\) with width \(w_1\).
These can represent the possible \(x\)-coordinates of all four-point sets in \(\ZZ^2\) with multi-width \((w_1,w_2)\).
The second width gives bounds on their possible \(y\)-coordinates and we can completely classify the four-point sets in the plane with multi-width \((1,w_2)\).
Similarly, these represent the possible first two coordinates of the vertices of tetrahedra of multi-width \((1,w_2,w_3)\).
By considering the possible \(z\)-coordinates we can assign to each point we obtain the classification above.

When \(w_1>1\) the number of cases which needs to be checked for this proof style increases dramatically.
However, the method can be used to create an algorithm which classifies the tetrahedra of a given multi-width.
In this way we begin classifying the width 2 case for small multi-width.
The number of tetrahedra classified can be found in Table~\ref{tab:width_2_tet}.
We take this classification only far enough to obtain and test Conjecture~\ref{conj:width_2_tet} on the number of multi-width \((2,w_2,w_3)\) tetrahedra.
The extent to which we can extend the two-dimensional results to the three-dimensional case remains open, but the similarities in the results we have found so far seem hopeful.

\begin{table}[ht]
    \centering
    \begin{tabular}{@{} cccccccccccc @{}}
    \headercell{} & \multicolumn{11}{c@{}}{\(w_3\)}\\
    \cmidrule(l){2-12}
    \(w_2\) & 2 & 3 & 4 & 5 & 6 & 7 & 8 & 9 & 10 & 11 & 12\\
    \midrule
    2 & 17 & 45 & 47 & 45 & 47 & 45 & 47 & 45 & 47 & 45 & 47\\
    3 & 0 & 87 & 178 & 175 & 178 & 175 & 178 & 175 & 178 & 175 & 178\\
    4 & 0 & 0 & 161 & 320 & 325 & 320 & 325 & 320 & 325 & 320 & 325\\
    5 & 0 & 0 & 0 & 244 & 493 & 490 & 493 & 490 & 493& 490 & 493\\
    6 & 0 & 0 & 0 & 0 & 358 & 716 & 721 & 716 & 721 & 716 & 721\\
    7 & 0 & 0 & 0 & 0 & 0 & 482 & 970 & 967 & 970 & 967 & 970\\
    8 & 0 & 0 & 0 & 0 & 0 & 0 & 636 & 1274 & 1279 & 1274 & 1279\\
    9 & 0 & 0 & 0 & 0 & 0 & 0 & 0 & 801 & 1609 & 1606 & 1609\\
    10& 0 & 0 & 0 & 0 & 0 & 0 & 0 & 0 & 995 & 1994 & 1999
    \end{tabular}
    \caption{The number of lattice tetrahedra with multi-width \((2,w_2,w_3)\) up to affine equivalence for small \(w_2\) and \(w_3\). For full list of tetrahedra see \cite{data}.}
    \label{tab:width_2_tet}
\end{table}

In Section~\ref{sec:width} we formally define the multi-width.
We prove some facts about a polytope of given multi-width.
In particular a 3-dimensional lattice polytope with multi-width \((w_1,w_2,w_3)\) is equivalent to a subset of \([0,w_1]\times [0,w_2] \times [0,w]\) where \(w\) is the smallest out of \(w_1+w_3-1\) and \(\max\{w_1+w_2,w_3\}\).
In Section~\ref{sec:quad} we classify the four-point sets in \(\ZZ\) with multi-width \(w_1\) and the four-point sets in \(\ZZ^2\) with multi-width \((w_1,w_2)\) which have \(x\)-coordinates \(\{0,0,0,w_1\}\) or \(\{0,0,w_1,w_1\}\).
A corollary of this is the classification of multi-width \((1,w_2)\) four-point sets.
In Section~\ref{sec:tet} we prove Theorem~\ref{thm:main}.
Propositions~\ref{prop:existence}, \ref{prop:tet_width_is} and \ref{prop:distinct} show that the map taking a lattice tetrahedron to its equivalence class is a well-defined bijection from \(\cS_{1,w_2,w_3}\) to \(\cT_{1,w_2,w_3}\).
In Section~\ref{sec:computational} we describe the computational extension of this classification.
We classify the multi-width \((w_1,w_2)\) four-point sets in the plane and the multi-width \((2,w_2,w_3)\) tetrahedra for small \(w_1\), \(w_2\) and \(w_3\).
Based on these classifications we make conjectures about the functions counting such sets and tetrahedra in general.

\subsection*{Acknowledgements}
I would like to thank my supervisors, Alexander Kasprzyk and Johannes Hofscheier, for generously sharing their expertise and supporting me throughout this project.
Also, thank you to anyone who asked if I could generalise the triangle classification; you inspired me to return to this project.
This research was supported in-part by the Heilbronn Institute for Mathematical Research.

\section{Width and parallelepipeds}
\label{sec:width}

Let \(N \cong \ZZ^d\) be a lattice, \(N^* \coloneqq \Hom(N,\ZZ)\cong \ZZ^d\) its dual lattice and \(N_{\RR} \coloneqq \RR \otimes_{\ZZ} N \cong \RR^d\) the real vector space containing \(N\).
For two tuples of integers \(w=(w_1,\dots,w_d)\) and \(w'=(w_1',\dots,w_d')\) we say that \(w <_{lex} w'\) when there is some \(1\leq i \leq d\) such that \(w_i < w_i'\) and \(w_j=w_j'\) for all \(j<i\). 
This defines the \emph{lexicographic order} on \(\ZZ^d\).
\begin{definition}
Let \(P\) be a lattice polytope and \(u \in N^*\) a dual vector.
We define the \emph{width of \(P\) with respect to \(u\)} to be
\[
\width_u(P) \coloneqq \max_{x \in P} \{u \cdot x\} - \min_{x \in P} \{u \cdot x\}.
\]
Since the widths are non-negative it is possible to define 
\[
\mwidth(P) \coloneqq \min_{u_1,\dots,u_d \in N^*}(\width_{u_1}(P),\dots, \width_{u_d}(P))
\]
where the minimum is taken with respect to lexicographic order and \(u_1,\dots, u_d\) are required to be linearly independent.
We call this tuple the \emph{multi-width} of \(P\) and call \(\width_{u_i}(P)\) the \emph{\(i\)-th width} of \(P\) written \(\width^i(P)\).
\end{definition}
We now define a polytope \(\cW_P \coloneqq (P-P)^*\) which is the dual of the Minkowski sum of \(P\) and \(-P\).
Notice that \(\cW_P\) is a rational polytope but need not be a lattice polytope.
This polytope encodes the widths of \(P\) in the following sense.
\begin{lemma}
\label{lem:width_poly}
For a dual lattice point \(u\), \(\width_u(P) \leq w\) if and only if \(u \in w\cW_P\).
\end{lemma}
\begin{proof}[Proof]
By definition, the dual polytope \(\cW_P\) is the set of rational points \(u\) such that \(u \cdot x \leq 1\) for all \(x \in P-P\).
For a fixed lattice point \(u\), \(\width_u(P) \leq w\) if and only if \(u \cdot (x_1-x_2) \leq w\) for all pairs of points \(x_1\) and \(x_2\) in \(P\).
This is equivalent to saying that \(\frac{1}{w}u \cdot x \leq 1\) for all points \(x\) of \(P-P\), or in other words \(u \in w\cW_P\).
\end{proof}
Note that this is equivalent to saying that the \(i\)-th width of \(P\) is the \(i\)-th successive minimum of \(\cW_P\).
For a definition of the successive minima of a polytope see \cite[p. 581]{KannanLovasz}.
We use Lemma~\ref{lem:width_poly} to prove the following result.
\begin{proposition}\label{prop:in_strip}
    Let \(d \geq 2\) and \(P \subset N_\RR\) be a lattice polytope.
    If \(P\) has widths \(w_1\) and \(w_2\) with respect to two linearly independent, primitive dual vectors, then \(P\) is equivalent to a subset of \([0,w_1] \times [0,w_2] \times \RR^{d-2}\).
\end{proposition}
\begin{proof}
Relabeling if necessary we may assume \(w_1\leq w_2\).
Pick linearly independent, primitive dual vectors \(u_1\) and \(u_2\) which realise the stated widths of \(P\).
We know that \(u_1,u_2 \in w_2\cW_P\).
As real vectors, \(u_1\) and \(u_2\) generate a two-dimensional vector space containing a sublattice of \(N^*\).
The triangle \(\conv(0,u_1,u_2) \subseteq w_2\cW_P\) contains a lattice point \(u_2'\) such that \(\{u_1,u_2'\}\) is a basis for this sublattice.
Since \(u_2' \in w_2\cW_P\) we know that \(\width_{u_2'}(P) \leq w_2\).
After a change of basis, we may assume that \(u_1\) and \(u_2'\) are the first two standard basis vectors.
This change of basis and a translation are sufficient to map \(P\) to a subset of \([0,w_1]\times[0,w_2]\times\RR^{d-2}\).
\end{proof}
Notice that this is an artefact of the first two dimensions since it relies on Pick's theorem to assume that all empty lattice triangles are equivalent.
In dimensions 3 and higher we may no longer assume that all empty simplices are equivalent.

\begin{proposition}\label{prop:containing_rectangle}
Let \(Q\) be a lattice polytope with widths \(w_1\), \(w_2\) and \(w_3\) in three linearly independent directions.
Assume that \(0<w_1 \leq w_2 \leq w_3\) then \(Q\) is equivalent to a subset of \([0,w_1] \times [0,w_2] \times [0,w_1+w_3-1]\).
Furthermore, if \((w_1,w_2,w_3)\) is the multi-width of \(Q\) then \(Q\) is equivalent to a subset of \([0,w_1] \times [0,w_2] \times [0,\max\{w_1+w_2,w_3\}].\)
\end{proposition}
\begin{proof}
By Proposition~\ref{prop:in_strip} we may assume that \(Q\) is a subset of the parallelepiped
\[
P \coloneqq \{v \in \RR^3: (1,0,0) \cdot v \in [0,w_1], (0,1,0) \cdot v \in [0,w_2], u \cdot v \in [a,a+w_3]\}
\]
for some integer \(a\) and some dual vector \(u\) linearly independent to \((1,0,0)\) and \((0,1,0)\).
Say \(u=(u_x,u_y,u_z)\) then \(u_z \neq 0\).
In fact we may assume \(u_z>0\), otherwise replace \(u\) with \(-u\) and adjust \(a\) so this does not change \(P\).
Therefore, we may pick integers \(k_x\) and \(k_y\) such that \(0 \leq k_iu_z-u_i < u_z\).
Now let \(\varphi\) be the shear described by
\[
(x,y,z) \mapsto (x,y,k_xx+k_yy+z).
\]
By inspecting the \(z\)-coordinates of the vertices of \(\varphi(P)\) we can show that
\begin{align*}
\width_{(0,0,1)}(\varphi(Q)) \leq & \frac{w_1(k_xu_z-u_x)+w_2(k_yu_z-u_y)+w_3}{u_z}\\
\leq & w_1\left(1-\frac{1}{u_z}\right)+w_2\left(1-\frac{1}{u_z}\right)+w_3\frac{1}{u_z}
\end{align*}
Since \(w_2 \leq w_3\) this is less that \(w_1+w_3\).
After a translation this shows that \(Q\) is equivalent to a subset of \([0,w_1]\times[0,w_2]\times[0,w_1+w_3-1]\).
This uses the fact that \(Q\) is a lattice polytope so has integral widths.

Now suppose the multi-width of \(Q\) is \((w_1,w_2,w_3)\) then we show that \(Q\) is equivalent to a subset of \([0,w_1] \times [0,w_2] \times [0,\max\{w_1+w_2,w_3\}]\).
In the above inequalities if \(u_z=1\) then \(\width_{(0,0,1)}(\varphi(Q)) \leq w_3\) so we are done.
If \(u_z \geq 2\) consider the fact that \(\width_{(0,0,1)}(\varphi(Q)) \leq w_1+w_2 + \frac{w_3-w_1-w_2}{u_z}\).
If \(w_3>w_1+w_2\) then \(w_1 + w_2 + \frac{w_3-w_1-w_2}{u_z}\) is at most \(\frac{w_1+w_2+w_3}{2}\) which is less than \(w_3\).
If \(w_3 \leq w_1+w_2\) then \(w_1+w_2+\frac{w_3-w_2-w_1}{u_z}\) is at most \(w_1+w_2\).
Since the third width of \(\varphi(Q)\) is \(w_3\) and its first two widths are realised by \((1,0,0)\) and \((0,1,0)\) it cannot have width less that \(w_3\) with respect to \((0,0,1)\).
This eliminates the case \(u_z \geq 2\) and \(w_3 > w_1+w_2\) and so, after a translation, \(\varphi(Q)\) is a subset of the desired box.
\end{proof}
This shows that any 3-dimensional lattice polytope with multi-width \((w_1,w_2,w_3)\) is equivalent to a subset of \([0,w_1]\times [0,w_2] \times [0,w]\) where 
\[
w \coloneqq \min\{w_1+w_3-1,\max\{w_1+w_2,w_3\}\}.
\]
This bound may not be sharp in general.

\section{Four-point sets in the plane}
\label{sec:quad}

The main aim of this section is to classify the four-point sets in the plane with first width \(1\).
The four-point sets with multi-width \((1,1)\) are just \(\{(0,0),(1,0),(0,1),(1,1)\}\) and \(\{(0,0),(0,0),(1,0),(0,1)\}\).
When \(w_2>1\) the four-point sets with multi-width \((1,w_2)\) are \(\{(0,0),(0,w_2),(0,y_0),(1,0)\}\) where \(y_0 \in [0,\frac{w_2}{2}]\) and \(\{(0,0),(0,w_2),(1,0),(1,y_1)\}\) where \(y_1 \in [0,w_2]\) (for example, see Figure~\ref{fig:4-pt_sets}).
This can be proven directly but here we will prove a more general result.
We will classify all four-point sets \(S\) in the plane with multi-width \((w_1,w_2)\) where \(w_2>w_1\) and if \(\width_{u_1}(S)=w_1\) then all points of \(S\) are contained in the two hyperplanes with normal vector \(u_1\) bounding \(S\).
This is sufficient to classify the width 1 four-point sets in the plane while being the most general classification which is practical to obtain with this method.
We do this because it may be useful towards a future extension of the tetrahedron classification.

\begin{figure}[b]
\centering
\begin{tikzpicture}[x=0.5cm,y=0.5cm]
\draw[fill=gray!30] (0,0) -- (0,4) -- (1,0) -- cycle;
\node[draw,circle,inner sep=1pt,fill] at (0,0) {};
\node[draw,circle,inner sep=1pt,fill] at (0,4) {};
\node[draw,circle,inner sep=1pt,fill] at (1,0) {};
\node[draw,circle,inner sep=2pt] at (0,0) {};

\draw[fill=gray!30] (0+2,0) -- (0+2,4) -- (1+2,0) -- cycle;
\node[draw,circle,inner sep=1pt,fill] at (0+2,0) {};
\node[draw,circle,inner sep=1pt,fill] at (0+2,4) {};
\node[draw,circle,inner sep=1pt,fill] at (1+2,0) {};
\node[draw,circle,inner sep=1pt,fill] at (0+2,1) {};

\draw[fill=gray!30] (0+4,0) -- (0+4,4) -- (1+4,0) -- cycle;
\node[draw,circle,inner sep=1pt,fill] at (0+4,0) {};
\node[draw,circle,inner sep=1pt,fill] at (0+4,4) {};
\node[draw,circle,inner sep=1pt,fill] at (1+4,0) {};
\node[draw,circle,inner sep=1pt,fill] at (0+4,2) {};

\draw[fill=gray!30] (0+6,0) -- (0+6,4) -- (1+6,0) -- cycle;
\node[draw,circle,inner sep=1pt,fill] at (0+6,0) {};
\node[draw,circle,inner sep=1pt,fill] at (0+6,4) {};
\node[draw,circle,inner sep=1pt,fill] at (1+6,0) {};
\node[draw,circle,inner sep=2pt] at (1+6,0) {};

\draw[fill=gray!30] (0+8,0) -- (0+8,4) -- (1+8,1) -- (1+8,0) -- cycle;
\node[draw,circle,inner sep=1pt,fill] at (0+8,0) {};
\node[draw,circle,inner sep=1pt,fill] at (0+8,4) {};
\node[draw,circle,inner sep=1pt,fill] at (1+8,0) {};
\node[draw,circle,inner sep=1pt,fill] at (1+8,1) {};

\draw[fill=gray!30] (0+10,0) -- (0+10,4) -- (1+10,2) -- (1+10,0) -- cycle;
\node[draw,circle,inner sep=1pt,fill] at (0+10,0) {};
\node[draw,circle,inner sep=1pt,fill] at (0+10,4) {};
\node[draw,circle,inner sep=1pt,fill] at (1+10,0) {};
\node[draw,circle,inner sep=1pt,fill] at (1+10,2) {};

\draw[fill=gray!30] (0+12,0) -- (0+12,4) -- (1+12,3) -- (1+12,0) -- cycle;
\node[draw,circle,inner sep=1pt,fill] at (0+12,0) {};
\node[draw,circle,inner sep=1pt,fill] at (0+12,4) {};
\node[draw,circle,inner sep=1pt,fill] at (1+12,0) {};
\node[draw,circle,inner sep=1pt,fill] at (1+12,3) {};

\draw[fill=gray!30] (0+14,0) -- (0+14,4) -- (1+14,4) -- (1+14,0) -- cycle;
\node[draw,circle,inner sep=1pt,fill] at (0+14,0) {};
\node[draw,circle,inner sep=1pt,fill] at (0+14,4) {};
\node[draw,circle,inner sep=1pt,fill] at (1+14,0) {};
\node[draw,circle,inner sep=1pt,fill] at (1+14,4) {};

\foreach \x in {-1,...,16}{
    \foreach \y in {-1,...,5}
        \node[cross=1.5pt] at (\x,\y) {};
}
\end{tikzpicture}
    \caption{All 4-point sets in \(\ZZ^2\) with multi-width \((1,4)\) up to affine equivalence and their convex hulls. Two points with the same coordinates are denoted by a circled dot.}
    \label{fig:4-pt_sets}
\end{figure}

First we classify all four-point sets in \(\ZZ\) of width \(w_1\).
\begin{proposition}\label{prop:marked_lines_class}
There is a bijection from the collection of lattice points in the triangle \(Q_{w_1} \coloneqq \conv((0,0),(0,w_1),(\frac{w_1}{2},\frac{w_1}{2}))\) to the set of the equivalence classes of the four-point sets in \(\ZZ\) with width \(w_1\).
It is given by the map taking \((x_1,x_2)\) to \(\{0,x_1,x_2,w_1\}\).
In particular, the number of such sets up to equivalence is
\[
\begin{cases}
\frac{w_1^2}{4} + w_1 + 1 & \text{if \(w_1\) is even}\\
\frac{w_1^2}{4} + w_1 + \frac34 & \text{if \(w_1\) is odd.}
\end{cases}
\]
\end{proposition}
\begin{proof}
The map \((x_1,x_2) \mapsto \{0,x_1,x_2,w_1\}\) is a well-defined map taking a lattice point of \(Q_{w_1}\) to a four-point set of width \(w_1\).
For surjectivity notice that the convex hull of any four-point set of width \(w_1\) is equivalent to \(\conv(0,w_1)\).
Therefore, we may assume that \(0\) and \(w_1\) are points in such a set and that \(x_1,x_2 \in [0,w_1]\) are the two remaining points.
By relabeling of the \(x_i\) we may assume that \(x_1 \leq x_2\).
A reflection takes \(\{0,x_1,x_2,w_1\}\) to \(\{0,w_1-x_2,w_1-x_1,w_1\}\) so we may assume that \(x_1 \leq w_1-x_2\).
This shows that \((x_1,x_2) \in Q_{w_1}\).

For injectivity let \((x_1,x_2)\) and \((x_1',x_2')\) be lattice points in \(Q_{w_1}\) such that \(\{0,x_1,x_2,w_1\}\) is equivalent to \(\{0,x_1',x_2',w_1\}\).
The only non-trivial affine automorphism of a line segment in \(\ZZ\) is the reflection about its midpoint so either \((x_1,x_2)=(x_1',x_2')\) or \((x_1,x_2) = (w_1-x_2',w_1-x_1')\).
In the first case we are done.
In the second case notice that \(x_1 = w_1-x_2' \geq x_1'\) and \(x_1' = w_1-x_2 \geq x_1\) so \(x_1=x_1'\).
Similarly \(x_2=x_2'\) which proves the result.

The counting can be seen by counting points in vertical lines of lattice points in \(Q_{w_1}\).
Thus there are \((w_1+1) + (w_1-1) + \dots + 1\) points in total if \(w_1\) is even and \((w_1+1) + (w_1-1) + \dots + 2\) if \(w_1\) is odd.
These simplify to the given formulas.
\end{proof}
We now move on to four-point sets in \(\ZZ^2\) with multi-width \((w_1,w_2)\).
If a multi-width \((w_1,\dots,w_d)\) polytope is a subset of a \(w_1 \times \dots \times w_d\) box it must have a vertex in each facet of this box otherwise it would have smaller multi-width.
Therefore, a multi-width \((w_1,w_2)\) four-point set which is a subset of \([0,w_1]\times[0,w_2]\) has \(x\)-coordinates equivalent to one of the above classified sets.
We restrict to the case where the corresponding point of \(Q_{w_1}\) is either \((0,0)\) or \((0,w_1)\) since this is sufficient to classify all multi-width \((1,w_2)\) four-point sets in \(\ZZ^2\).
We additionally assume that \(w_1<w_2\), since the four-point sets with multi-width \((1,1)\) are easy to identify and for \(w_1>1\) classifying the multi-width \((w_1,w_1)\) four-point sets adds unnecessary complexity to the proofs.
\begin{proposition}\label{prop:width1_quad_1}
    Let \(S\) be a four-point set in the plane with multi-width \((w_1,w_2)\) where \(0<w_1<w_2\).
    There is a dual vector \(u_1\) such that \(\width_{u_1}(S) = w_1\) and \(u_1 \cdot S\) is equivalent to \(\{0,0,0,w_1\}\) if and only if \(S\) is equivalent to exactly one of the following four-point sets:
    \begin{itemize}
        \item \(\{(0,0), (0,w_2),(0,y_0), (w_1,y_1)\}\) where \(0 \leq y_0 < \frac{w_2}{2}\) and \(0 \leq y_1 < w_1\),
        \item \(\{(0,0), (0,w_2), (0,\frac{w_2}{2}), (w_1,y_1)\}\) where \(0\leq y_1 \leq (w_2-y_1 \mod w_1)\) and \(w_2\) is even.
    \end{itemize}
\end{proposition}
\begin{proof}
First we show that the listed four-point sets have multi-width \((w_1,w_2)\).
It is enough to notice that in either case if \(u=(u_x,u_y)\) is a dual vector with \(u_y \neq 0\), then
\[
\width_u(S) \geq |u\cdot (0,w_2) - u \cdot(0,0)| = |u_yw_2| \geq w_2.
\]
The image of these sets under \(u_1=(1,0)\) is \(\{0,0,0,w_1\}\) which proves the implication in one direction.

Next we show that all four-point sets \(S\) with multi-width \((w_1,w_2)\) and a dual vector \(u_1\) such that \(u_1 \cdot S\) is equivalent to \(\{0,0,0,w_1\}\) are equivalent to one of the two given cases.
Let \(S\) be such a set, then by Proposition~\ref{prop:in_strip} we may assume it is a subset of \([0,w_1] \times [0,w_2]\).
Since \(w_1<w_2\) the direction in which \(S\) has width \(w_1\) is unique up to sign so \(u_1=\pm(1,0)\).
Under \(u_1\) points of \(S\) are mapped to (possibly \(-1\) times) their \(x\)-coordinates so the only way for these to map to something equivalent to \(\{0,0,0,w_1\}\) is to have three points of \(S\) on one vertical edge of the rectangle and the fourth point on the other vertical edge.
Therefore, possibly after the reflection \((x,y) \mapsto (w_1-x,y)\), we may assume that \(S\) contains three points with \(x\)-coordinate \(0\) and one with \(x\)-coordinate \(w_1\).
Also, \(S\) must contain \((0,0)\) and \((0,w_2)\) otherwise, by a shear \((x,y) \mapsto (x,y-kx)\) for some integer \(k\), \(S\) is equivalent a subset of a smaller rectangle which contradicts the widths.
Therefore, we assume that \(S=\{(0,0),(0,w_2),(0,y_0),(w_1,y_1)\}\).

By the reflection \((x,y) \mapsto (x,w_2-y)\) we may assume that \(0 \leq y_0 \leq \frac{w_2}{2}\).
By a shear \((x,y) \mapsto (x,y-kx)\) for some integer \(k\) we may assume that \(0 \leq y_1 < w_1\).
If \(y_0 = \frac{w_2}{2}\) and \(y_1 > (w_2-y_1 \mod w_1)\) then we can make the \(y\)-coordinate of the vertex on \(x=w_1\) smaller by a reflection in the line \(y=\frac{w_2}{2}\) followed by a shear.
In more precise terms, pick \(k\) such that \(w_2-y_1 -kx = (w_2-y_1 \mod w_1)\) then the reflection and shear \((x,y) \mapsto (x,w_2-y-kx)\) takes \(S\) to one of the given four-point sets.
This proves that \(S\) is equivalent to a set of one of the given forms.

Finally we show that the four-points sets in the two cases are unique.
Suppose
\[
S=\{(0,0),(0,w_2),(0,y_0),(w_1,y_1)\} \sim \{(0,0),(0,w_2),(0,y_0'),(w_1,y_1')\} = S'
\]
where \(S\) and \(S'\) are each of either of the forms from the proposition.
We will show that \(S\) and \(S'\) are equal.
We can think of these sets as their convex hulls, which are triangles, with a marked point.
Since \(w_2>w_1\), considering the lattice length of line segments (i.e. the number of lattice points they contain minus 1) in \([0,w_1]\times[0,w_2]\), we see that the edge from \((0,0)\) to \((0,w_2)\) is the only edge of each triangle with lattice length \(w_2\).
Therefore, an affine map taking \(S\) to \(S'\) must map this edge back to itself.
This reduces us to shears \((x,y) \mapsto (x,y-kx)\) and the reflection followed by a shear \((x,y) \mapsto (w_2-y-kx)\) where \(k\) is an integer.
The images of \(S\) under such maps are
\begin{align*}
    &\{(0,0),(0,w_2),(0,y_0),(w_1,y_1-kw_1)\},\quad\text{and}\\ 
    &\{(0,0),(0,w_2),(0,w_2-y_0),(w_1,w_2-y_1-kw_1)\}
\end{align*}
Since \(0 \leq y_0,y_0' \leq \frac{w_2}{2}\) this shows that \(y_0=y_0'\).
The \(y\)-coordinate of the fourth points of these images must equal \(y_1'\) so, since \(0 \leq y_1' < w_1\), we must always choose \(k\) so that this \(y\)-coordinate is reduced modulo \(w_1\).
Since \(0 \leq y_1 < w_1\) if the shear takes \(S\) to \(S'\) this means that \(y_1=y_1'\) and \(S=S'\).
If instead the reflection followed by a shear takes \(S\) to \(S'\) then \(y_0' = w_2-y_0\) and since \(y_0=y_0'\) we have \(y_0=\frac{w_2}{2}\).
This means that \(y_1 \leq (w_2-y_1 \mod w_1) = y_1'\) and by the symmetric argument exchanging \(S\) and \(S'\) we have \(y_1' \leq (w_2-y_1' \mod w_1) = y_1\) so \(y_1=y_1'\) and \(S=S'\).
\end{proof}
\begin{proposition}\label{prop:width1_quad_2}
    Let \(S\) be a four-point set in the plane with multi-width \((w_1,w_2)\) where \(0<w_1<w_2\).
    There is a dual vector \(u_1\) such that \(\width_{u_1}(S) = w_1\) and \(u_1 \cdot S\) is equivalent to \(\{0,0,w_1,w_1\}\) if and only if \(S\) is equivalent to exactly one of the following four-point sets:
    \begin{itemize}
        \item \(\{(0,0), (0,w_2), (w_1,y_1), (w_1,y_2)\}\) where \(0 \leq y_1 \leq y_2 \leq w_2\) and \\\(y_1 \leq (w_2-y_2 \mod w_1)\)
        \item \(\{(0,0), (0,y_0), (w_1,y_1), (w_1,w_2)\}\) where \(\max\{w_2-y_1, w_2-(w_1-y_1)\}\leq y_0 < w_2\)
    \end{itemize}
\end{proposition}
\begin{proof}
First we show that the listed four-point sets have multi-width \((w_1,w_2)\).
The first case follows by the same proof as that in Proposition~\ref{prop:width1_quad_1} since it still contains \((0,0)\) and \((0,w_2)\).
In the second case, it suffices to show that for any dual vector \(u=(u_x,u_y)\) with \(u_y>0\), \(\width_u(S) \geq w_2\).
The image of \(S\) under \(u\) is 
\[
u \cdot S = \{0, u_yy_0, u_xw_1+u_yy_1, u_xw_1+u_yw_2\}.
\]
Suppose for contradiction that the width of \(S\) with respect to \(u\) is less than \(w_2\).
This means that the difference of any two elements in \(u \cdot S\) must be less than \(w_2\) so then \(u_xw_1+u_yw_2 < w_2\) which implies \(u_x < 0\).
Also, \(u_yy_0-u_xw_1-u_yy_1<w_2\) which we rearrange to show
\begin{align}\label{eqn:limit}
u_x > (u_yy_0-w_2-u_yy_1)/w_1.
\end{align}
By the conditions on \(y_0\) and \(y_1\) we have \(y_0-y_1 \geq w_2-w_1\) so \(u_y(y_0-y_1) \geq w_2-w_1\).
Combining this with \eqref{eqn:limit} shows that \(u_x >-1\) which is the desired contradiction.
Under \((1,0)\) these sets are taken to \(\{0,0,w_1,w_1\}\) which shows the implication in one direction.

Let \(S\) be a four-point set with multi-width \((w_1,w_2)\) with a dual vector \(u_1\) such that \(u_1\cdot S\) is equivalent to \(\{0,0,w_1,w_1\}\).
We will show that \(S\) is equivalent to one of the sets listed in the proposition.
By Proposition~\ref{prop:in_strip} we may assume this is a subset of \([0,w_1]\times [0,w_2]\).
Since \(w_1 < w_2\) the direction in which the first width is realised is unique up to sign so \(u_1=\pm(1,0)\).
This shows that \(S\) contains two points with \(x\)-coordinate \(0\) and two with \(x\)-coordinate \(w_1\).
Also, \(S\) must contain points with \(y\)-coordinates \(0\) and \(w_2\) or it would have smaller multi-width.
This means \((0,0)\) or \((w_1,0)\) is in \(S\) and \((0,w_2)\) or \((w_1,w_2)\) is in \(S\).
Possibly after applying the reflection \((x,y) \mapsto (w_1-x,y)\) we may assume that \((0,0) \in S\).
We may assume one of the following two possibilities: 
\begin{enumerate}
    \item[A.] \(S=\{(0,0),(0,w_2),(w_1,y_1),(w_1,y_2)\}\) where \(0 \leq y_1 \leq y_2 \leq w_2\),
    \item[B.] \(S=\{(0,0),(0,y_0),(w_1,y_1),(w_1,w_2)\}\) where \(0 \leq y_0<w_2\) and \(0<y_1\leq w_2\)
\end{enumerate}
where the additional inequalities come from relabeling the \(y_i\) and removing overlap between the cases.

In A we aim to minimise the \(y\)-coordinates of vertices on the line \(x=w_1\) and can do so either by a shear about the \(y\)-axis or by a reflection in the line \(y=\frac{w_2}{2}\) followed by such a shear.
We may assume by a shear that \(0 \leq y_1 <w_1\).
Consider the map given by \((x,y) \mapsto (x,w_2-y-kx)\) for the integer \(k\) such that \(w_2-y_2-kw_1 = (w_2-y_2\mod w_1)\).
This map is self inverse and takes \(S\) to \(\{(0,0),(0,w_2),(w_1,y_1'),(w_1,y_2')\}\) which is also of the form A and \(y_1' <w_1\).
Either \(y_1 \leq (w_2-y_2 \mod w_1) = y_1'\) or \(y_1' \leq (w_2-y_2' \mod w_1) = y_1\).
In either case \(S\) is equivalent to one of the sets listed in the proposition.

In B we can get a set of the same form by applying the map \((x,y) \mapsto (w_1-x,w_2-y)\) which replaces \(y_0\) and \(y_1\) with \(w_2-y_1\) and \(w_2-y_0\) respectively.
Therefore, we may assume that \(y_0 \geq w_2-y_1\).
Now consider the image \((-1,1) \cdot S =\{0,y_0,y_1-w_1,w_2-w_1\}\).
To prevent the width of \(S\) with respect to \((-1,1)\) being less that \(w_2\) the difference between some pair of elements of \((-1,1) \cdot S\) must be at least \(w_2\).
However, checking these differences case by case only \(y_0-(w_2-w_1)\) and \(y_0-(y_1-w_1)\) can be at least \(w_2\).
If \(y_0-(w_2-w_1) \geq w_2\) then definitely \(y_0-(y_1-w_1) \geq w_2\) so we may assume the latter holds.
Therefore, \(S\) is equivalent to one of the sets listed in the proposition.

Now we show that the sets listed in the proposition are distinct up to equivalence.
The two cases are distinct since the convex hull of the first case has an edge of lattice length \(w_2\) and the second does not.
Let \(S\) and \(S'\) be of the first form and suppose
\[
S=\{(0,0),(0,w_2),(w_1,y_1),(w_1,y_2)\} \sim \{(0,0),(0,w_2),(w_1,y_1'),(w_1,y_2')\}=S'.
\]
Either these are both equal to the vertices of \([0,w_1] \times [0,w_2]\) or their convex hulls each have exactly one edge of lattice length \(w_2\).
If they are equal we are done, otherwise the map taking \(S\) to \(S'\) must preserve the line segment from \((0,0)\) to \((0,w_2)\).
This reduces us to shears \((x,y) \mapsto (x,y-kx)\) and the reflection followed by a shear \((x,y) \mapsto (x,w_2-y-kx)\) for integers \(k\).
Since \(0 \leq y_1,y_1' <w_1\) if a shear maps \(S\) to \(S'\) then \(S=S'\).
If the reflection followed by a shear maps \(S\) to \(S'\) then \(y_1' = (w_2-y_2 \mod w_1) \geq y_1\) and symmetrically \(y_1 \geq y_1'\) so \(y_1=y_1'\).
Since the volume of the convex hulls of \(S\) and \(S'\) must be equal this shows that \(y_2=y_2'\) and \(S=S'\).

Now let \(S\) and \(S'\) be of the second form listed in the proposition and suppose
\[
S=\{(0,0),(0,y_0),(w_1,y_1),(w_1,w_2)\} \sim \{(0,0),(0,y_0'),(w_1,y_1'),(w_1,w_2)\}=S'.
\]
By the multi-width of these we know that \(\pm(1,0)\) are the only dual vectors under which they have width \(w_1\).
Therefore, a map taking the convex hull of \(S\) to the convex hull of \(S'\) must take edges with normal \((1,0)\) to edges with normal \((1,0)\).
Thus it suffices to consider the lengths of the vertical edges of the convex hulls of \(S\) and \(S'\).
By the conditions on \(S\) and \(S'\) their left-most vertical edge is at least as long as their right most vertical edge so \(y_0=y_0'\), \(y_1=y_1'\) and \(S=S'\).
\end{proof}

We use these propositions to classify the four-point sets with multi-width \((1,w_2)\).
\begin{corollary}\label{cor:width_1_fps}
Let \(S\) be a four-point set in \(\ZZ^2\) with multi-width \((1,w_2)\) then if \(w_2>1\) either \(S\) is equivalent to 
\begin{itemize}
    \item \(\{(0,0),(0,w_2),(0,y_0),(1,0)\}\) with \(y_0 \in [0,\frac{w_2}{2}]\) or
    \item \(\{(0,0),(0,w_2),(1,0),(1,y_1)\}\) with \(y_1 \in [0,w_2]\).
\end{itemize}
Counting the possible integers \(y_0\) and \(y_1\) shows that these are counted by
\[
\begin{cases}
\frac{3w_2}{2} + 2 & \text{if \(w_2\) even}\\
\frac{3w_2}{2} +\frac32 & \text{if \(w_2\) odd.}
\end{cases}
\]
If instead \(w_2=1\) then \(S\) is equivalent to 
\begin{itemize}
    \item \(\{(0,0),(0,1),(1,0),(1,1)\}\) or 
    \item \(\{(0,0),(0,0),(1,0),(0,1)\}\).
\end{itemize}
\end{corollary}

\section{Proof of Theorem~\ref{thm:main}}
\label{sec:tet}

In this section we prove Theorem~\ref{thm:main}, that is we show that the map taking a tetrahedron to its affine equivalence class defines a bijection from the set of tetrahedra \(\cS_{1,w_2,w_3}\) to the set \(\cT_{1,w_2,w_3}\) of tetrahedra of multi-width \((1,w_2,w_3)\) up to affine equivalence.

We begin with a preparatory lemma towards surjectivity.
\begin{lemma}\label{lem:existence}
Let \(T\) be a lattice tetrahedron with multi-width \((1,w_2,w_3)\).
Then there exists a tetrahedron \(T'\) of type~\ref{item:tet_0001}, \ref{item:tet_0011_0}, \ref{item:tet_0011_1} or \ref{item:tet_0011_2}, as in Definition~\ref{def:tet_class}, which is equivalent to \(T\).
\end{lemma}
\begin{proof}
By Proposition~\ref{prop:containing_rectangle} we may assume that \(T\) is a subset of \([0,1] \times [0,w_2] \times [0,w_3]\).
The vertices of \(T\) can only be in the planes \(x=0\) and \(x=1\) so, possibly after a reflection, we may assume the \(x\)-coordinates of the vertices of \(T\) are either \(\{0,0,0,1\}\) or \(\{0,0,1,1\}\).
We will show that
\begin{itemize}
    \item If the \(x\)-coordinates of \(T\) are \(\{0,0,0,1\}\) then \(T\) is equivalent to a type~\ref{item:tet_0001} tetrahedron,
    \item If the \(x\)-coordinates of \(T\) are \(\{0,0,1,1\}\) then \(T\) is equivalent to a type~\ref{item:tet_0011_0}, \ref{item:tet_0011_1} or \ref{item:tet_0011_2} tetrahedron.
\end{itemize}

If the \(x\)-coordinates are \(\{0,0,0,1\}\) then \(T\) is the convex hull of a triangle embedded in the plane \(x=0\) and a point with \(x\)-coordinate 1.
For integers \(k_1\) and \(k_2\), the shears \((x,y,z) \mapsto (x,y+k_1x,z+k_2x)\), can take the vertex with \(x\)-coordinate \(1\) to any lattice point with \(x\)-coordinate \(1\) without changing the triangle in the plane \(x=0\).
Therefore, we may assume that, under the projection onto the last two coordinates, \(T\) is mapped to the triangle.
The triangle is a subset of a \(w_2 \times w_3\) rectangle so its multi-width is at most \((w_2,w_3)\).
If it had multi-width lexicographically smaller than \((w_2,w_3)\) there would be dual vectors \(u_2'\) and \(u_3'\) linearly independent from \(u_1=(1,0,0)\) such that \((1,\width_{u_2'}(T),\width_{u_3'}(T)) <_{lex} (1,w_2,w_3)\) which is a contradiction.
Therefore, this triangle has multi-width \((w_2,w_3)\).
By \cite[Theorem 1.1]{triangles} we can assume the triangle is in \(\cS_{w_2,w_3}\).
Then by another shear of the same form we can move the fourth vertex to \((1,0,0)\) which proves that \(T\) is equivalent to a tetrahedron of the form~\ref{item:tet_0001}.

If the \(x\)-coordinates are \(\{0,0,1,1\}\) we need to consider the four-point set we get by projecting the vertices of \(T\) onto the first two coordinates.
This must have multi-width \((1,w_2)\) otherwise \(T\) has smaller multi-width.
By Corollary~\ref{cor:width_1_fps} the set is equivalent to one of the sets \(\{(0,0),(0,w_2),(1,0),(1,y_1)\}\) for an integer \(y_1 \in [0,w_2]\) and by an affine map on \(T\) we may assume they are equal.
It remains to determine the \(z\)-coordinates of the vertices of \(T\).
These are integers in \([0,w_3]\).
At least one of them must equal \(0\) and one \(w_3\) otherwise \(T\) has width less than \(w_3\) with respect to \((0,0,1)\) contradicting the multi-width.

There are 12 ways to assign \(0\) and \(w_3\) to two of the vertices.
See Figure~\ref{fig:12_cases} for the full list.
By a reflection in the plane \(z=\frac{w_3}{2}\), as depicted in Figure~\ref{fig:case_reduction}, we may swap which vertices are assigned \(0\) and \(w_3\).
In this way we see that \subref{fig:g}-\subref{fig:l} are equivalent to \subref{fig:a}-\subref{fig:f} so we disregard the last 6 cases.

We are left with cases \subref{fig:a}-\subref{fig:f}.
By a reflection in the plane \(y=\frac{w_2}{2}\) followed by the shear \((x,y,z) \mapsto (x,y-(w_2-y_1)x,z)\), as depicted in Figure~\ref{fig:case_reduction}, we can swap the \(z\)-coordinates assigned to the lower two vertices with those assigned to the upper two vertices.
In this way we see that \subref{fig:d} and \subref{fig:e} are equivalent to \subref{fig:c} and \subref{fig:b} respectively so we may disregard \subref{fig:d} and \subref{fig:e} also.

We are left with cases \subref{fig:a}, \subref{fig:b}, \subref{fig:c} and \subref{fig:f}.
In case \subref{fig:a}, after a shear about the plane \(x=0\), we may assume that \(z_1=w_3\) or \(z_2=w_3\).
Therefore, \subref{fig:a} is included in case \subref{fig:b} or \subref{fig:c} and we may disregard \subref{fig:a}.
Similarly, in case \subref{fig:f}, after a shear about the plane \(x=1\) we may assume that \(z_1=0\) or \(z_2=0\).
Therefore, \subref{fig:f} is included in case \subref{fig:c} or case \subref{fig:e} and thus \subref{fig:b}.
We disregard \subref{fig:f} as a result.

We are left with only cases \subref{fig:b} and \subref{fig:c}.
We will show that we can also disregard case \subref{fig:c} by showing that its width is incorrect unless it is also equivalent to a type \subref{fig:b} tetrahedron.
In case \subref{fig:c}, if \(y_1=0\), \(z_1=0\) or \(z_2=w_3\) then these tetrahedra would be included in case \subref{fig:b} (or \subref{fig:e} and thus \subref{fig:b}) so we assume these three equalities are false.
\begin{figure}
\centering
\subfloat[]{\label{fig:a}
\begin{tikzpicture}[x=0.5cm,y=0.5cm]
\draw[fill=gray!40] (0,0)--(0,3)--(1,2)--(1,0) -- cycle;
\node[anchor=east] at (0,0) {0};
\node[anchor=east] at (0,3) {\(w_3\)};
\node[anchor=west] at (1,0) {\(z_1\)};
\node[anchor=west] at (1,2) {\(z_2\)};
\end{tikzpicture}}
\subfloat[]{\label{fig:b}
\begin{tikzpicture}[x=0.5cm,y=0.5cm]
\draw[fill=gray!40] (0,0)--(0,3)--(1,2)--(1,0) -- cycle;
\node[anchor=east] at (0,0) {0};
\node[anchor=east] at (0,3) {\(z_1\)};
\node[anchor=west] at (1,0) {\(w_3\)};
\node[anchor=west] at (1,2) {\(z_2\)};
\end{tikzpicture}}
\subfloat[]{\label{fig:c}
\begin{tikzpicture}[x=0.5cm,y=0.5cm]
\draw[fill=gray!40] (0,0)--(0,3)--(1,2)--(1,0) -- cycle;
\node[anchor=east] at (0,0) {0};
\node[anchor=east] at (0,3) {\(z_1\)};
\node[anchor=west] at (1,0) {\(z_2\)};
\node[anchor=west] at (1,2) {\(w_3\)};
\end{tikzpicture}}
\subfloat[]{\label{fig:d}
\begin{tikzpicture}[x=0.5cm,y=0.5cm]
\draw[fill=gray!40] (0,0)--(0,3)--(1,2)--(1,0) -- cycle;
\node[anchor=east] at (0,0) {\(z_1\)};
\node[anchor=east] at (0,3) {0};
\node[anchor=west] at (1,0) {\(w_3\)};
\node[anchor=west] at (1,2) {\(z_2\)};
\end{tikzpicture}}
\subfloat[]{\label{fig:e}
\begin{tikzpicture}[x=0.5cm,y=0.5cm]
\draw[fill=gray!40] (0,0)--(0,3)--(1,2)--(1,0) -- cycle;
\node[anchor=east] at (0,0) {\(z_1\)};
\node[anchor=east] at (0,3) {\(0\)};
\node[anchor=west] at (1,0) {\(z_2\)};
\node[anchor=west] at (1,2) {\(w_3\)};
\end{tikzpicture}}
\subfloat[]{\label{fig:f}
\begin{tikzpicture}[x=0.5cm,y=0.5cm]
\draw[fill=gray!40] (0,0)--(0,3)--(1,2)--(1,0) -- cycle;
\node[anchor=east] at (0,0) {\(z_1\)};
\node[anchor=east] at (0,3) {\(z_2\)};
\node[anchor=west] at (1,0) {0};
\node[anchor=west] at (1,2) {\(w_3\)};
\end{tikzpicture}}\\
\subfloat[]{\label{fig:g}
\begin{tikzpicture}[x=0.5cm,y=0.5cm]
\draw[fill=gray!40] (0,0)--(0,3)--(1,2)--(1,0) -- cycle;
\node[anchor=east] at (0,0) {\(w_3\)};
\node[anchor=east] at (0,3) {0};
\node[anchor=west] at (1,0) {\(z_1\)};
\node[anchor=west] at (1,2) {\(z_2\)};
\end{tikzpicture}}
\subfloat[]{\label{fig:h}
\begin{tikzpicture}[x=0.5cm,y=0.5cm]
\draw[fill=gray!40] (0,0)--(0,3)--(1,2)--(1,0) -- cycle;
\node[anchor=east] at (0,0) {\(w_3\)};
\node[anchor=east] at (0,3) {\(z_1\)};
\node[anchor=west] at (1,0) {0};
\node[anchor=west] at (1,2) {\(z_2\)};
\end{tikzpicture}}
\subfloat[]{\label{fig:i}
\begin{tikzpicture}[x=0.5cm,y=0.5cm]
\draw[fill=gray!40] (0,0)--(0,3)--(1,2)--(1,0) -- cycle;
\node[anchor=east] at (0,0) {\(w_3\)};
\node[anchor=east] at (0,3) {\(z_1\)};
\node[anchor=west] at (1,0) {\(z_2\)};
\node[anchor=west] at (1,2) {0};
\end{tikzpicture}}
\subfloat[]{\label{fig:j}
\begin{tikzpicture}[x=0.5cm,y=0.5cm]
\draw[fill=gray!40] (0,0)--(0,3)--(1,2)--(1,0) -- cycle;
\node[anchor=east] at (0,0) {\(z_1\)};
\node[anchor=east] at (0,3) {\(w_3\)};
\node[anchor=west] at (1,0) {0};
\node[anchor=west] at (1,2) {\(z_2\)};
\end{tikzpicture}}
\subfloat[]{\label{fig:k}
\begin{tikzpicture}[x=0.5cm,y=0.5cm]
\draw[fill=gray!40] (0,0)--(0,3)--(1,2)--(1,0) -- cycle;
\node[anchor=east] at (0,0) {\(z_1\)};
\node[anchor=east] at (0,3) {\(w_3\)};
\node[anchor=west] at (1,0) {\(z_2\)};
\node[anchor=west] at (1,2) {0};
\end{tikzpicture}}
\subfloat[]{\label{fig:l}
\begin{tikzpicture}[x=0.5cm,y=0.5cm]
\draw[fill=gray!40] (0,0)--(0,3)--(1,2)--(1,0) -- cycle;
\node[anchor=east] at (0,0) {\(z_1\)};
\node[anchor=east] at (0,3) {\(z_2\)};
\node[anchor=west] at (1,0) {\(w_3\)};
\node[anchor=west] at (1,2) {0};
\end{tikzpicture}}
    \caption{Image \(\conv((0,0),(0,w_2),(1,0),(1,y_1))\) of a tetrahedron under projection onto the first two coordinates. Labels denote the \(z\)-coordinates at each vertex. Numbers \(z_1\) and \(z_2\) are integers in the range \([0,w_3]\) and these twelve cases include all affine equivalence classes of tetrahedra with multi-width \((1,w_2,w_3)\) and \(x\)-coordinates \(\{0,0,1,1\}\) in \([0,1] \times [0,w_2] \times [0,w_3]\).}
    \label{fig:12_cases}
\end{figure}
\begin{figure}
\centering
\begin{tikzpicture}[x=0.75cm,y=0.75cm]
\draw[fill=gray!40] (0,0)--(0,3)--(1,2)--(1,0) -- cycle;
\node[anchor=east] at (0,0) {\(z_1\)};
\node[anchor=east] at (0,3) {\(z_2\)};
\node[anchor=west] at (1,0) {\(z_3\)};
\node[anchor=west] at (1,2) {\(z_4\)};

\draw[fill=gray!40] (0+5,0)--(0+5,3)--(1+5,2)--(1+5,0) -- cycle;
\node[anchor=east] at (0+5,0) {\(w_3-z_1\)};
\node[anchor=east] at (0+5,3) {\(w_3-z_2\)};
\node[anchor=west] at (1+5,0) {\(w_3-z_3\)};
\node[anchor=west] at (1+5,2) {\(w_3-z_4\)};

\draw[fill=gray!40] (0+10,0)--(0+10,3)--(1+10,2)--(1+10,0) -- cycle;
\node[anchor=east] at (0+10,0) {\(z_2\)};
\node[anchor=east] at (0+10,3) {\(z_1\)};
\node[anchor=west] at (1+10,0) {\(z_4\)};
\node[anchor=west] at (1+10,2) {\(z_3\)};
\end{tikzpicture}
    \caption{Image of \(T=\conv((0,0,z_1),(0,w_2,z_2),(1,0,z_3),(1,y_1,z_4))\) and two equivalent tetrahedra under projection onto the first two coordinates. Labels denote the \(z\)-coordinates at each vertex. The second tetrahedron is obtained from the first by a reflection in the plane \(z=\frac{w_3}{2}\). The third is obtained from the first by a reflection in the plane \(y=\frac{w_2}{2}\) followed by the shear \((x,y,z) \mapsto (x,y-(w_2-y_1)x,z)\).}
    \label{fig:case_reduction}
\end{figure}
The images of the vertices of \subref{fig:c} under \((-z_2,0,1)\), \((0,-1,1)\) and \((w_3-y_1,1,-1)\) are 
\[
\{0,z_1,z_2,w_3-z_2\},\quad \{0,z_1-w_2,z_2,w_3-y_1\} \quad \text{and} \quad \{0,w_2-z_1,w_3-y_1-z_2,0\}
\]
respectively.
If any of these is a subset of \([0,w_3)\) then \(T\) would have width less than \(w_3\) in some direction linearly independent to \((1,0,0)\) and \((0,1,0)\) which is a contradiction.
To prevent this all three of the following points must be true.
\begin{itemize}
    \item \(z_1=w_3\) or \(z_2=w_3\) or \(z_2=0\)
    \item \(z_1<w_2\) or \(z_2=w_3\) or \(y_1=0\)
    \item (\(w_2=w_3\) and \(z_1=0\)) or \(z_1>w_2\) or \(y_1=z_2=0\) or \(y_1+z_2>w_3\)
\end{itemize}
Eliminating options we have assumed are not true we see that it is impossible to satisfy all of these conditions at once.
Therefore, we may discard \subref{fig:c} entirely and assume that when the \(x\)-coordinates of our tetrahedron are \(\{0,0,1,1\}\) then it is equivalent to
\[
T = \conv((0,0,0),(0,w_2,z_1),(1,0,w_3),(1,y_1,z_2))
\]
for some integers \(z_1\) and \(z_2\) in \([0,w_3]\).
We will now refine this general tetrahedron into a type~\ref{item:tet_0011_0}, \ref{item:tet_0011_1} or \ref{item:tet_0011_2} tetrahedron considering the cases \(y_1=0\) and \(y_1> 0\) separately.

If \(y_1=0\) then the image of the vertices of \(T\) under \((-z_2,0,1)\) and \((-z_2,-1,1)\) are
\[
\{0,z_1,w_3-z_2,0\}, \quad \text{and} \quad \{0,z_1-w_2,w_3-z_2,0\}
\]
respectively.
Neither of these can be a subset of \([0,w_3)\) as this would contradict the widths of \(T\).
Therefore, \(z_2=0\) since otherwise we would need both \(z_1=w_3\) and \(z_1 <w_2\) which is false.
By a shear \((x,y,z) \mapsto (x,y,z-ky)\) and possibly the reflection in the plane \(y=\frac{w_2}{2}\) we may assume \(z_1 \leq (-z_1 \mod w_2)\) and so \(0 \leq z_1 \leq \frac{w_2}{2}\). This shows that \(T\) is of the form \ref{item:tet_0011_0}.

If instead \(y_1>0\) consider the images of the vertices of \(T\) under \((-z_2,0,1)\), \((-1,1,1)\) and \((-1,-1,1)\) which are
\[
\{0,z_1,w_3-z_2,0\},\quad \{0,w_2+z_1,w_3-1,y_1+z_2-1\}, \quad \text{and} \quad \{0,z_1-w_2,w_3-1,z_2-y_1-1\}
\]
respectively.
Again, none of these can be a subset of \([0,w_3)\) so all three of the following points must be true:
\begin{itemize}
    \item \(z_1=w_3\) or \(z_2=0\)
    \item \(w_2+z_1 \geq w_3\) or \(y_1+z_2-1 \geq w_3\)
    \item \(z_1 < w_2\) or \(z_2<y_1+1\).
\end{itemize}
To satisfy these we must have either \(z_1=w_3\) and \(z_2\leq y_1\) or \(z_2=0\) and \(z_1 \geq w_3-w_2\).
The tetrahedron when \(z_1=w_3\) and \(z_2=y_1\) is equivalent but not equal to that when \(z_2=0\) and \(z_1=w_3-w_2\) by the shear \((x,y,z) \mapsto (x,y,z-y)\).
Therefore, we may also assume that \(z_2<y_1\) to avoid duplicates.
This proves that \(T\) is of the form \ref{item:tet_0011_1} or \ref{item:tet_0011_2}.
\end{proof}

The following shows that the map taking a tetrahedron to its affine equivalence class gives a surjective map from \(\cS_{1,w_2,w_3}\) to \(\cT_{1,w_2,w_3}\).
\begin{proposition}\label{prop:existence}
Let \(T\) be a lattice tetrahedron with multi-width \((1,w_2,w_3)\).
Then there exists some \(T' \in \cS_{1,w_2,w_3}\) which is equivalent to \(T\).
\end{proposition}
\begin{proof}
By Lemma~\ref{lem:existence} and the definition of \(\cS_{1,w_2,w_3}\), if \(1 < w_2 <w_3\) then we are done.

Now we consider the special cases.
If \(w_2=w_3\) and \(T\) is a tetrahedron of the form \ref{item:tet_0011_1} then the image of \(T\) under the map \((x,y,z) \mapsto (1-x,z_1-z+x(w_2-z_1),y)\) is
\[
\conv((0,0,0), (0,w_2,y_1), (1,0,w_2), (1,z_1,0)).
\]
If \(y_1 > z_1\) this is also of the form \ref{item:tet_0011_1} but with the roles of \(y_1\) and \(z_1\) swapped.
Therefore, to remove duplicates we may assume that \(y_1 \leq z_1\).
If \(w_2=w_3\) and \(T\) is a tetrahedron of the form \ref{item:tet_0011_2} then the image of \(T\) under the map \((x,y,z) \mapsto (x,w_2-z,w_2-y)\) is
\[
\conv((0,0,0),(0,w_2,w_2),(1,0,w_2),(1,w_2-z_1,w_2-y_1)).
\]
This is also of the form~\ref{item:tet_0011_2} and the map is self inverse so unless \(T\) is equal to its image we need to eliminate one of these tetrahedra to remove duplicates.
Our \(T\) is equal to its image only when \(z_1+y_1=w_2\) so we may assume that \(z_1 \leq w_2-y_1\).

When \(w_2=1\) substituting into tetrahedra~\ref{item:tet_0001}-\ref{item:tet_0011_2} and simplifying reduces us to the following cases:
\begin{itemize}
    \item \(\conv((0,0,0),(0,0,w_3),(0,1,0), (1,0,0))\)
    \item \(\conv((0,0,0),(0,1,0),(1,0,0),(1,0,w_3))\)
    \item \(\conv((0,0,0),(0,1,z_1),(1,0,w_3),(1,1,0))\) where \(z_1=w_3-1\) or \(z_1=w_3\).
\end{itemize}
Under \((x,y,z) \mapsto (1-x-y,y,z)\) the second of these maps to the first so we are left with the three tetrahedra appearing in \(\cS_{1,1,w_3}\).
Finally, when \(w_3=1\) two of these are equivalent to the convex hull of the empty triangle embedded in the plane \(x=0\) and the point \((1,0,0)\) so it reduces to the two cases in \(\cS_{1,1,1}\).
\end{proof}
The following proves that the map taking a tetrahedron to its affine equivalence class gives a map from \(\cS_{1,w_2,w_3}\) to \(\cT_{1,w_2,w_3}\).
\begin{proposition}\label{prop:tet_width_is}
Let \(T \in \cS_{1,w_2,w_3}\), then the multi-width of \(T\) is \((1,w_2,w_3)\).
\end{proposition}
\begin{proof}
By definition, the tetrahedra in \(\cS_{1,w_2,w_3}\) are always of one of the forms \ref{item:tet_0001}-\ref{item:tet_0011_2}.
Therefore, it suffices to show that a tetrahedron satisfying one of these conditions has multi-width \((1,w_2,w_3)\) for any \(w_3 \geq w_2 \geq 1\).
Let \(T\) be in one of these forms.
Lattice polytopes have integral widths and only have width zero in some direction if their dimension is less than that of the space they are in.
Since \(\width_{(1,0,0)}(T)=1\) and \(T\) has non-zero volume its first width is 1.
Furthermore, \(T\) has widths \(w_2\) and \(w_3\) realised by the dual vectors \((0,1,0)\) and \((0,0,1)\) respectively.

There are two ways in which the remaining two widths can fail.
Either there is a dual vector \(u\) linearly independent to \((1,0,0)\) such that \(\width_u(T) < w_2\) or there is a dual vector \(u\) linearly independent to \(\{(1,0,0),(0,1,0)\}\) such that \(\width_u(T) < w_3\).
To prove these do not occur it suffices to show that for all \(u=(u_x,u_y,0)\) with \(u_y \neq 0\) \(\width_u(T) \geq w_2\) and for all \(u=(u_x,u_y,u_z)\) with \(u_z \neq 0\) \(\width_u(T) \geq w_3\).

Let \(\pi\) be the projection onto the first two coordinates then
\[
\width_{(u_x,u_y,0)}(T) = \width_{(u_x,u_y)}(\pi(T)).
\]
However, the four-point set which is the image of the vertices of \(T\) under \(\pi\) has multi-width \((1,w_2)\) by Corollary~\ref{cor:width_1_fps}.
The first width of this set is realised by \((1,0)\) therefore, for all \(u=(u_x,u_y,0)\) with \(u_y \neq 0\), we have \(\width_u(T) = \width_{(u_x,u_y)}(\pi(T)) \geq w_2\).

Now suppose for contradiction there exists dual vector \(u=(u_x,u_y,u_z)\) with \(u_z\neq0\) is such that \(\width_u(T)<w_3\).
Without loss of generality we may assume that \(u_z >0\).
Then by the proof of Proposition~\ref{prop:containing_rectangle} a map of the form \((x,y,z) \mapsto (x,y,k_1x+k_2y+z)\) takes \(T\) to a subset of a \(1 \times w_2 \times \width_u(T)\) box for some integers \(k_1\) and \(k_2\).
The image of \(T\) under this map is one of the following corresponding to the form of \(T\).
\begin{enumerate}
    \item \(\conv(\{0\} \times \left(\begin{smallmatrix}1 & 0\\ k_2 & 1
    \end{smallmatrix}\right)t, (1,0,k_1))\) where \(t \in \cS_{w_2,w_3}\)
    \item \(\conv((0,0,0),(0,w_2,z_1+k_2w_2),(1,0,k_1),(1,0,w_3+k_1))\) where \(0 \leq z_1 \leq \frac{w_2}{2}\)
    \item \(\conv((0,0,0),(0,w_2,z_1+k_2w_2),(1,0,w_3+k_1),(1,y_1,k_1+k_2y_1))\) where \(0 < y_1 \leq w_2\) and \(w_3-w_2 \leq z_1 \leq w_3\),
    \item \(\conv((0,0,0),(0,w_2,w_3+k_2w_2),(1,0,w_3+k_1),(1,y_1,z_1+k_1+k_2y_1))\) where \(0 < y_1 < w_2\) and \(0 < z_1 < y_1\).
\end{enumerate}
We will show that it is impossible for any of these to have width less than \(w_3\) with respect to \((0,0,1)\).

1. Let \(\pi\) be the projection onto the last two coordinates then for a polytope \(P\) we have
\[
\width_{(0,0,1)}(P) = \width_{(0,1)}(\pi(P)).
\]
This means that \(\conv(\{0\} \times \left(\begin{smallmatrix}1 & 0\\ k_2 & 1 \end{smallmatrix}\right)t, (1,0,k_1))\) where \(t \in \cS_{w_2,w_3}\) never has width less than \(w_3\) with respect to \((0,0,1)\) thanks to the widths of \(t\).

2. The tetrahedron of the form~\ref{item:tet_0011_0} has width at least \(w_3\) with respect to \((0,0,1)\) due to the vertices \((1,0,k_1)\) and \((1,0,w_3+k_1)\).

3. In a tetrahedron of the form~\ref{item:tet_0011_1} if the width with respect to \((0,0,1)\) was less than \(w_3\) the difference between every pair of \(z\)-coordinates must be less than \(w_3\).
In particular we would need to have \(w_3+k_1-k_1-k_2y_1 < w_3\) and \(z_1+k_2w_2<w_3\).
The first of these implies that \(k_2>0\) which combines with the second to show that \(z_1+w_2 < w_3\).
However, \(z_1 \geq w_3-w_2\) which is a contradiction.

4. Similarly, in a tetrahedron of the form~\ref{item:tet_0011_2} we would need \(w_3+k_2w_2 < w_3\) and \(w_3+k_1-z_1-k_1-k_2y_1<w_3\).
The first of these implies that \(k_2<0\) and the second implies that \(k_2>-z_1/y_1>-1\) which is a contradiction.
\end{proof}

For the proof of Proposition~\ref{prop:distinct} we will need the following lemma.
\begin{lemma}\label{lem:y1_is_min}
    Let \(1 < w_2 \leq w_3\) and let \(T\) be a tetrahedron in \(\cS_{1,w_2,w_3}\) whose \(x\)-coordinates are \(\{0,0,1,1\}\).
    Suppose, the projection of \(T\) onto the first two coordinates is 
    \[
    \conv((0,0),(0,w_2),(1,0),(1,y_1)).
    \]
    Then, for any surjective lattice homomorphism \(\pi: \ZZ^3 \to\ZZ^2\), if \(\pi(T)\) is equivalent to 
    \[
    \conv((0,0),(0,w_2),(1,0),(1,y_1'))
    \]
    for some integer \(y_1' \in [0,w_2]\) then \(y_1' \geq y_1\).
    In other words, \(y_1\) is minimal.
\end{lemma}
\begin{proof}
For tetrahedra of the form~\ref{item:tet_0011_0} this is immediate since \(y_1=0\).

Let \(T\) be of the form~\ref{item:tet_0011_1} or \ref{item:tet_0011_2}.
Let \(\pi: \ZZ^3 \to \ZZ^2\) be a surjective lattice homomorphism such that \(\pi(T)\) is equivalent to \(\conv((0,0),(0,w_2),(1,0),(1,y_1'))\) for an integer \(y_1' \in [0,w_2]\).
Let \(P = (p_{ij})\) be the \(2 \times 3\) integral matrix defining \(\pi\).
By Proposition~\ref{prop:in_strip} there are dual vectors \(u_1\) and \(u_2 \in (\ZZ^2)^*\) which form a basis of \((\ZZ^2)^*\) and which realise the first two widths of \(\pi(T)\).
This means that \(\width_{u_i}(\pi(T)) = \width_{u_i \circ \pi}(T) = w_i\) for \(i=1\) and \(2\).
Since \(w_2>1\) the direction in which \(P\) has width \(1\) is unique so, possibly after changing the sign of \(u_1\), we have \(u_1P = (1,0,0)\).
Let \(U\) be the matrix with rows \(u_1\) and \(u_2\) then replace \(P\) with \(UP\) and change \(\pi\) accordingly.
In this way we can assume that the first row of \(P\) is \((1,0,0)\).
This does not alter our previous assumptions about \(\pi\) since this operation is a unimodular map in \(\ZZ^2\).

If \(w_2 < w_3\) a vector realising the second width of \(P\) must be of the form \((u_x,u_y,0)\) so the final entry of \(u_2P\) is \(0\).
This allows us to assume \(p_{23}=0\).
Now we have image
\[
    \pi(T) = \conv((0,0),(0,p_{22}w_2),(1,p_{21}),(1,p_{21}+p_{22}y_1)).
\]
This has an edge of lattice length \(|p_{22}w_2|\) so \(p_{22}\) must be \(0\) or \(\pm1\).
If \(p_{22}=0\) then \(\pi(T)\) is just a line segment, contradicting our assumptions on \(\pi\).
If \(p_{22}=\pm1\) then \(y_1' = |p_{22}y_1+p_{21}-p_{21}| = y_1\) so \(y_1' \geq y_1\) as desired.

It remains to consider the case \(w_3=w_2>1\).
By replacing \(P\) with \(\left(\begin{smallmatrix}
    1 & 0 \\ -p_{21} & 1
\end{smallmatrix}\right)P\) we may assume \(p_{21}=0\).
This leaves us with the following two cases corresponding to \ref{item:tet_0011_1} and \ref{item:tet_0011_2}
\[
\pi(T) = \conv((0,0),(0,p_{22}w_2+p_{23}z_1),(1,p_{23}w_2),(1,p_{22}y_1) \quad \text{or}
\]
\[
\pi(T) = \conv((0,0), (0,p_{22}w_2+p_{23}w_2), (1,p_{23}w_2), (1,p_{22}y_1+p_{23}z_1)).
\]
The dual vector \((1,0)\) is the unique dual vector realising the first width of both \(\pi(T)\) and \(\conv((0,0),(0,w_2),(1,0),(1,y_1'))\).
These quadrilaterals each have (at most) two facets with normal vector \((1,0)\) so these facets must be equivalent.
This means that one of the vertical edges of \(\pi(T)\) must have length \(w_2\) and the other must have length in \([0,w_2]\).
Therefore, to complete the proof it suffices to show that for each choice of vertical edge to be length \(w_2\) the other vertical edge must have length at least \(y_1\).

First consider the quadrilateral associated to case~\ref{item:tet_0011_1}.
If \(|p_{22}w_2+p_{23}z_1|=w_2\) then, after a possible change of sign of the second line of \(P\), we may assume that \(p_{23}=w_2(1-p_{22})/z_1\).
Then \(|p_{23}w_2-p_{22}y_1| = |w_2^2/z_1-p_{22}(w_2^2/z_1+y_1)|\).
This is the modulus of a linear function in \(p_{22}\) so to find the smallest values it takes we notice that it is zero when \(p_{22} = w_2^2/(w_2^2 + y_1z_1) \in [0,1]\). 
Since \(p_{22}\) is an integer the length is smallest when either \(p_{22}=0\) or \(1\) which gives values \(w_2^2/z_1\) and \(y_1\) respectively.
Both of these are at least \(y_1\) given the conditions on a tetrahedron of the form~\ref{item:tet_0011_1} when \(w_2=w_3\).

On the other hand, if \(|p_{23}w_2-p_{22}y_1| = w_2\) then as above we may assume that \(p_{22} = w_2(p_{23}-1)/y_1\).
Then \(|p_{22}w_2+p_{23}z_1| = |-w_2^2/y_1+p_{23}(w_2^2/y_1+z_1)|\) which is zero when \(p_{23} = w_2^2/(w_2^2+z_1y_1) \in [0,1]\).
Since \(p_{23}\) is an integer it is actually smallest at either \(w_2^2/z_1\) or \(z_1\) both of which are at least \(y_1\) given our assumptions.

Now for the quadrilateral associated to case~\ref{item:tet_0011_2}.
If \(|p_{22}w_2+p_{23}w_2| = w_2\) then as above we may assume that \(p_{22}+p_{23} = 1\).
Then \(|p_{23}(w_3-z_1)-p_{22}y_1| = |p_{23}(w_2-z_1+y_1)-y_1|\) which is zero when \(p_{23} = y_1/(w_2-z_1+y_1) \in [0,1]\).
Since \(p_{23}\) is an integer it is actually smallest at either \(w_2-z_1\) or \(y_1\) both of which are at least \(y_1\) given our assumptions.

On the other hand, if \(|p_{23}(w_2-z_1)-p_{22}y_1| = w_2\) then notice that \(|p_{22}w_2+p_{23}w_2| = |p_{22}+p_{23}|w_2\).
Since these are all integers this is at least \(y_1\) unless \(p_{23} = -p_{22}\).
However, then we would have \(w_2 = |p_{23}|(w_2-z_1+y_1)\). 
We can't let both \(p_{22}\) and \(p_{23}\) be zero so then \(w_2 \geq w_2-z_1+y_1 > w_2\) which is a contradiction.
\end{proof}

The following shows injectivity of the map taking a tetrahedron to its affine equivalence class from \(\cS_{1,w_2,w_3}\) to \(\cT_{1,w_2,w_3}\).
\begin{proposition}\label{prop:distinct}
Tetrahedra in \(\cS_{1,w_2,w_3}\) are distinct under affine unimodular maps.
\end{proposition}
\begin{proof}
The two tetrahedra in \(\cS_{1,1,1}\) have normalised volumes \(1\) and \(2\).
Since volume is an affine invariant they must be distinct.
If \(w_3>1\) the three tetrahedra in \(\cS_{1,1,w_3}\) have normalised volumes \(w_3\), \(2w_3-1\) and \(2w_3\) so must also be distinct.

In the remaining cases our goal is to show that if \(T\) and \(T'\) are equivalent tetrahedra in \(\cS_{1,w_2,w_3}\) then either \(T=T'\) or this leads to a contradiction.
We have \(w_2>1\) so up to sign \(u_1=(1,0,0)\) is the unique vector such that each tetrahedron has width \(1\) with respect to \(u_1\).
Let \(T\) and \(T'\) be equivalent tetrahedra in \(\cS_{1,w_2,w_3}\) then the image of the vertices of \(T\) and \(T'\) under \(u_1\) must be equivalent.
In other words, the set of \(x\)-coordinates of two equivalent tetrahedra in \(\cS_{1,w_2,w_3}\) is also equivalent.
Therefore, tetrahedra of the form~\ref{item:tet_0001} are always distinct from the others.
Furthermore, if \(T\) and \(T'\) are equivalent tetrahedra of the form \ref{item:tet_0001} then they each have a unique facet with normal \(u_1\), so these facets must be equivalent too.
These facets were triangles in \(\cS_{w_2,w_3}\) so by \cite[Theorem 1.1]{triangles} this means \(T=T'\).

Now let \(T\) and \(T'\) be equivalent tetrahedra in \(\cS_{1,w_2,w_3}\) of the forms~\ref{item:tet_0011_0}, \ref{item:tet_0011_1} or \ref{item:tet_0011_2}.
By Lemma~\ref{lem:y1_is_min} if we project the sets of vertices of \(T\) and \(T'\) onto the first two coordinates we get the same set.
That is the \(x\)- and \(y\)-coordinates of \(T\) and \(T'\) are the same.
This immediately tells us that tetrahedra of the form~\ref{item:tet_0011_0} are distinct from those of the forms~\ref{item:tet_0011_1} and \ref{item:tet_0011_2}.
For the rest we will show the following four facts:
\begin{enumerate}
    \item[A.] If both \(T\) and \(T'\) are of the form~\ref{item:tet_0011_0} then \(T=T'\)
    \item[B.] If both \(T\) and \(T'\) are of the form~\ref{item:tet_0011_1} then \(T=T'\)
    \item[C.] If both \(T\) and \(T'\) are of the form~\ref{item:tet_0011_2} then \(T=T'\)
    \item[D.] If \(T\) is of the form~\ref{item:tet_0011_1} and \(T'\) is of the form~\ref{item:tet_0011_2} then we get a contradiction.
\end{enumerate}

A. Unless \(w_2=w_3\) and \(z_1=z_1'=0\) the only edge of \(T\) or \(T'\) with lattice length \(w_3\) is the one from \((1,0,0)\) to \((1,0,w_3)\).
Therefore, either \(T=T'\) or the affine map taking \(T\) to \(T'\) preserves this edge.
There are only four ways to map vertices of \(T\) to \(T'\) satisfying this.
Define the map \(\theta\) by \((x,y,z) \mapsto (x,y,z,1)\) and let \(\pi : \ZZ^4 \to \ZZ^3\) be the projection onto the first three coordinates.
Affine maps in \(\ZZ^3\) are exactly the maps \((x,y,z) \mapsto \pi(\theta(x,y,z)U)\) where \(U\) is a unimodular matrix with last column \((0,0,0,1)^T\).
Let \(M\) and \(M'\) be the matrices whose rows are the vertices of \(\theta(T)\) and \(\theta(T')\) respectively.
Let \(\sigma M\) denote the matrix obtained by permuting the rows of \(M\) according to \(\sigma\).
At least one of \(M^{-1}M'\), \(((12)M)^{-1}M'\), \(((34)M)^{-1}M'\) or \(((12)(34)M)^{-1}M'\) is unimodular.
These matrices are
\begin{align*}
\begin{pmatrix}
    1 & 0 & 0 & 0\\
    0 & 1 & 0 & 0\\
    0 & (z_1'-z_1)/w_2 & 1 & 0\\
    0 & 0 & 0 & 1
\end{pmatrix},\quad
&\begin{pmatrix}
    1 & 0 & 0 & 0\\
    -w_2 & -1 & 0 & w_2\\
    -z_1' & -(z_1+z_1')/w_2 & 1 & z_1'\\
    0 & 0 & 0 & 1
\end{pmatrix}\\
\begin{pmatrix}
    1 & 0 & 0 & 0\\
    0 & 1 & 0 & 0\\
    w_3 & (z_1 + z_1')/w_2 & -1 & 0\\
    0 & 0 & 0 & 1
\end{pmatrix},\quad
&\begin{pmatrix}
    1 & 0 & 0 & 0\\
    -w_2 & -1 & 0 & w_2\\
    w_3 -z_1' & (z_1-z_1')/w_2 & -1 & z_2\\
    0 & 0 & 0 & 1
\end{pmatrix}
\end{align*}
Therefore, either \((z_1-z_1')/w_2\) or \((z_1+z_1')/w_2\) is an integer.
In either case, the fact that \(0 \leq z_1 ,z_1' \leq \frac{w_2}{2}\) forces \(T=T'\).

B. The normalised volume of \(T\) and \(T'\) is \(w_2w_3+z_1y_1 = w_2w_3+z_1'y_1\), therefore \(z_1=z_1'\) and \(T=T'\).

C. The normalised volume of \(T\) and \(T'\) is \(w_2w_3-w_2z_1+y_1w_3 = w_2w_3-w_2z_1'+y_1w_3\), therefore \(z_1=z_1'\) and \(T=T'\).

D. Let \(P\) and \(P'\) be the parallelograms obtained by intersecting \(2T\) and \(2T'\) with the plane \(x=1\).
These are:
\begin{align*}
    &P=\conv((0,w_3),(y_1,0),(w_2,w_3+z_1),(w_2+y_1,z_1))\\
    &P'=\conv((0,w_3),(y_1,z_1'),(w_2,2w_3),(w_2+y_1,w_3+z_1'))
\end{align*}
Since \(T\) and \(T'\) are equivalent so are \(P\) and \(P'\).
We will show that this leads to a contradiction.
By the symmetries of a parallelogram, \(P-(0,w_3)\) must be unimodular equivalent to either \(P'-(0,w_3)\) or \(P'-(y_1,z_1')\).
Consider three matrices \(M, M_1\) and \(M_2\) whose rows are the following vertices of \(P-(0,w_3)\), \(P'-(0,w_3)\) and \(P'-(y_1,z_1')\) adjacent to the origin:
\[
M = \begin{pmatrix}
    y_1 & -w_3 \\ w_2 & z_1
\end{pmatrix}, \quad
M_1 = \begin{pmatrix}
    y_1 & z_1'-w_3 \\ w_2 & w_3
\end{pmatrix}, \quad
M_2 = \begin{pmatrix}
    -y_1 & w_3-z_1' \\ w_2 & w_3
\end{pmatrix}.
\]
One of \(M_1^{-1}M, ((12)M_1)^{-1}M, M_2^{-1}M\) and \(((12)M_2)^{-1}M\) must be a unimodular matrix. 
The necessary inverses are
\begin{align*}
M_1^{-1}=\frac{1}{V}
\begin{pmatrix}
    w_3 & w_3-z_1' \\ -w_2 & y_1
\end{pmatrix}, \quad
&M_2^{-1}=\frac{1}{V}
\begin{pmatrix}
    -w_3 & w_3-z_1' \\ w_2 & y_1
\end{pmatrix}\\
((12)M_1)^{-1}=\frac{1}{V}
\begin{pmatrix}
    w_3-z_1' & w_3 \\ y_1 & -w_2
\end{pmatrix}, \quad
&((12)M_2)^{-1}=\frac{1}{V}
\begin{pmatrix}
    w_3-z_1' & -w_3 \\ y_1 & w_2
\end{pmatrix}
\end{align*}
where \(V=w_2w_3-w_2z_1'+w_3y_1\).
Using these we know that either
\begin{align*}
&U=M_1^{-1}M=\frac{1}{V}\begin{pmatrix}
    w_2w_3 - w_2z_1' + w_3y_1 & -w_3^2 + w_3z_1 - z_1z_1'\\
    0  & w_2w_3 + z_1y_1
\end{pmatrix}
\end{align*}
is a unimodular matrix or one of the following is an integer
\[
\frac{w_2w_3-y_1z_1'+w_3y_1}{V},\quad \frac{z_1y_1-w_2w_3}{V}, \quad \frac{-w_2w_3+w_3y_1-z_1'y_1}{V}.
\]
These are entries \((1,1)\), \((2,2)\) and \((1,1)\) of \(((12)M_1)^{-1}M, M_2^{-1}M\) and \(((12)M_2)^{-1}M\) respectively.
Since the diagonal entries of \(U\) are positive they must both be \(1\) for \(U\) to be unimodular.
From this notice that
\[
\begin{pmatrix}
y_1 & z_1'-w_3\\ w_2 & w_3
\end{pmatrix}
\begin{pmatrix}
1 & a \\ 0 & 1
\end{pmatrix} = M_1U = M = 
\begin{pmatrix}
y_1 & -w_3 \\ w_2 & z_1
\end{pmatrix}
\]
for some integer \(a\).
From this we show that \(z_1'-w_3-ay_1=-w_3\) and \(w_3-aw_2=z_1\).
Since \(w_3-w_2 \leq z_1 \leq w_3\) either \(a=0\) or \(1\).
Therefore, \(z_1'=0\) or \(y_1\) both of which are contradictory so \(U\) cannot be unimodular.

Notice that \(-z_1'y_1 > -w_2z_1'\) so the first of the fractions is at least 2.
From this we show that \(2w_2z_1' \geq w_2w_3 + w_3y_1 + z_1'y_1\) which is a contradiction since \(w_2w_3\) and \(w_3y_1\) are both greater than \(w_2z_1'\) and \(z_1'y_1\) is non-negative.
Since \(y_1 < w_2\) and \(z_1 \leq w_3\) the second fraction is at most \(-1\) from which we show \(w_2z_1' \geq z_1y_1+w_3y_1\).
This is contradictory since \(w_3 \geq w_2\) and \(y_1 > z_1'\).
Finally, since \(y_1 < w_2\) the third fraction is also at most \(-1\) so \(z_1'(w_2+y_1) \geq 2w_3y_1\).
This is a contradiction since \(z_1'<y_1\) and \(w_2+y_1<2w_2 \leq 2w_3\).
\end{proof}

We now use the results of this section to prove Theorem~\ref{thm:main}.

\begin{proof}[Proof of Theorem~\ref{thm:main}]
Proposition~\ref{prop:tet_width_is} shows that the map taking a tetrahedron to its equivalence class is a well-defined map from \(\cS_{1,w_2,w_3}\) to \(\cT_{1,w_2,w_3}\).
Propositions~\ref{prop:existence} and \ref{prop:distinct} show that it is bijective.
It remains to find the cardinality of \(\cS_{1,w_2,w_3}\).
When \(w_2=1\) this is immediate.
When \(w_2>1\) we combine the triangles classification with the new tetrahedra to get the desired counts.
\end{proof}

The generating function of the sequence counting lattice triangles with second width \(w_2\) was the Hilbert series of a hypersurface in a weighted projective space so we investigate the generating function of the cardinality of \(\cT_{1,w_2,w_3}\) denoted \(|\cT_{1,w_2,w_3}|\).
\begin{corollary}
    The generating function of \(|\cT_{1,w_2,w_3}|\) is
    \[
    \sum_{w_2=1}^\infty \sum_{w_3=w_2}^\infty t^{w_2}s^{w_3}|\cT_{1,w_2,w_3}|
    = \frac{f(s,t)}{2(1-s^2)(1-ts)^3(1+ts)}
\]
where \(f(s,t)\) is the polynomial
\begin{align*}
    & t^5s^7 + 5t^5s^6 + 4t^5s^5 - 2t^4s^7 - 5t^4s^6 - 9t^4s^5 - 4t^4s^4 + 4t^3s^6 + 13t^3s^5 + 7t^3s^4\\& - 6t^3s^3 - 5t^2s^4 - 3t^2s^3 + 4t^2s^2 - 4ts^4 - 10ts^3 - 4ts^2 + 2ts + 2s^3 + 14s^2 + 20s + 8.
\end{align*}
\end{corollary}
\begin{proof}[Sketch of proof]
This can be shown using Theorem~\ref{thm:main}.
First note that \(\sum_{w_3=w_2}^\infty s^{w_3} |\cT_{1,w_2,w_3}|\) is
\[
    \frac{s(s+2)}{1-s}
\]
when \(w_2=1\) and
\[
    s^{w_2}\frac{w_2^2(s+1)^2 + w_2(1-s^2) + (\frac32s^2 + \frac52s + \frac32) + \frac12(-1)^{w_2}(s^2 + s + 1)}{1-s^2}.
\]
otherwise.
Combining these we get the desired result.
This can all be done by hand using facts about the generating functions of polynomials or with the assistance of computer algebra.
\end{proof}
To instead count lattice tetrahedra with first width \(1\) and third width \(w_3\) we let \(t=1\) in the above generating function resulting in
\[
    \frac{- s^7 + 4s^6 + 8s^5 - 6s^4 - 3s^3 + 22s + 8}{2(1-s)^4(1+s)^2}.
\]
Neither of these generating functions share any of the properties of the one from triangles.
However, this does not prevent a function counting lattice tetrahedra of a given multi-width in general from doing so.

\section{Computational and conjectural results}
\label{sec:computational}

The above method can be adapted into an algorithm which classifies four-point sets and tetrahedra of a given multi-width.
We implement all algorithms using \textsc{Magma~V2.27}.

Algorithm~\ref{alg:quad} classifies four-point sets of multi-width \((w_1,w_2)\).
It takes the list of four-point sets in the line of width \(w_1\) and assigns a \(y\)-coordinate in the range \([0,w_2]\) to each point of each set in every possible way.
The resulting sets in the plane include all four-point sets of multi-width \((w_1,w_2)\).

We eliminate any which do not have the correct widths.
Let \(P\) be the convex hull of such a set then we use the polytope \(\cW_P\) to check its multi-width.
The \(i\)-th width of \(P\) is \(w\) if and only if the dimension of \(\conv((w-1) \cW_P\cap \ZZ^d)\) is less than \(i\) and the dimension of \(\conv(w \cW_P\cap \ZZ^d)\) is at least \(i\).
This allows us to check if the multi-width of a polytope is equal to \((w_1,w_2)\) without necessarily calculating its multi-width.

We also discard repeated sets using an affine unimodular normal form.
Kreuzer and Skarke introduced a unimodular normal form for lattice polytopes in their \textsc{Palp} software \cite{PALP}.
This can be extended to an affine normal form by translating each vertex of a polytope to the origin in turn and finding the minimum unimodular normal form among these possibilities.
If the convex hull of a four-point set is a quadrilateral we can use this normal form without adjustment.
If the convex hull is a triangle we find the normal form of this triangle then consider the possible places the fourth point can be mapped to in this normal form.
We choose the minimum such point and call the set of vertices of the triangle and this point the normal form of the four-point set denoted \(\NF(S)\).
Note that we keep the normal forms of each set as well as the set itself as we need the four-point sets written as a subset of \([0,w_1] \times [0,w_2]\) for the tetrahedra classification.

\begin{algorithm}
\DontPrintSemicolon
\KwData{The set \(\cP\) of all lattice points in \(Q_{w_1} = \conv((0,0),(0,w_1),(\frac{w_1}{2},\frac{w_1}{2})\).}
\KwResult{The set \(\cA\) containing all four-point sets in the plane with multi-width \((w_1,w_2)\) written as a subset of \([0,w_1] \times [0,w_2]\).}
\(\cA \longleftarrow \emptyset\)\;
NormalForms \(\longleftarrow \emptyset\)\;
\For{\((x_1,x_2) \in \cP\)}{
    \For{\(h_1,h_2,h_3,h_4 \in [0,w_2]\cap\ZZ\) such that \(h_i=0\) and \(h_j=w_2\) for some \(i<j\)}{
        \(S \longleftarrow \{(0,h_1),(x_2,h_2),(x_3,h_3),(w_1,h_4)\}\)\;
        \If{\(\mwidth(S) = (w_1,w_2)\) and \(\NF(S) \notin\) \emph{NormalForms}}{
            \(\cA \longleftarrow \cA \cup \{S\}\)\;
            NormalForms \(\longleftarrow\) NormalForms \(\cup\) \(\{\NF(S)\}\)\;
            
        }
    }
}
\caption{Classifying the four-point sets in \(\ZZ^2\) with multi-width \((w_1,w_2)\).}
\label{alg:quad}
\end{algorithm}

Running Algorithm~\ref{alg:quad} for small widths produces Table~\ref{tab:quad}.
We use this data to give estimates for a function counting the width \((w_1,w_2)\) four-point sets in the plane.
To decide how much data to compute we use the following.
\begin{proposition}\label{prop:four_pt_set_bound}
There are at most \((\frac{w_1^2}{4} + w_1 + c)(6w_2^2+1)\) four-point sets in the plane with multi-width \((w_1,w_2)\) where \(c=1\) if \(w_1\) is even and \(c=\frac34\) if \(w_1\) is odd.
\end{proposition}
\begin{proof}
By Proposition~\ref{prop:marked_lines_class} we know that there are 
\[
\begin{cases}
\frac{w_1^2}{4} + w_1 + 1 & \text{if \(w_1\) even}\\
\frac{w_1^2}{4} + w_1 +\frac34 &\text{if \(w_1\) odd}
\end{cases}
\]
four-point sets in the line with width \(w_1\).
Let \((y_1,\dots,y_4) \in [0,w_2]^4\) be a lattice point representing the \(y\)-coordinates we give to each point.
We know that there exist indices \(i_0\) and \(i_1\) such that \(y_{i_0}=0\) and \(y_{i_1}=w_2\).
By a reflection we may assume that \(i_0 < i_1\).
We also assume these are as small as possible.
Counting the possibilities in each of the six cases we show that there are at most \(6w_2^2+1\) ways to assign \(y\)-coordinates to a four-point set in the line.
\end{proof}
\begin{table}[ht]
    \centering
    \begin{tabular}{@{} ccccccccccccc @{}}
    \headercell{} & \multicolumn{12}{c@{}}{\(w_2\)}\\
    \cmidrule(l){2-13}
    \(w_1\) & 1 & 2 & 3 & 4 & 5 & 6 & 7 & 8 & 9 & 10 & 11 & 12\\
    \midrule
    1 & 2 & 5 & 6 & 8 & 9 & 11 & 12 & 14 & 15 & 17 & 18 & 20\\
    2 & 0 & 13 & 31 & 42 & 49 & 60 & 67 & 78 & 85 & 96 & 103 & 114\\
    3 & 0 & 0 & 39 & 101 & 123 & 148 & 170 & 195 & 217 & 242 & 264 & 289   \\
    4 & 0 & 0 & 0 & 114 & 282 & 342 & 394 & 454 & 506 & 566 & 618 & 678\\
    5 & 0 & 0 & 0 & 0 & 254 & 624 & 727 & 835 & 938 & 1046 & 1149 & 1257\\
    6 & 0 & 0 & 0 & 0 & 0 & 520 & 1239 & 1428 & 1605 & 1794 & 1971 & 2160\\
    7 & 0 & 0 & 0 & 0 & 0 & 0 & 937 & 2206 & 2490 & 2781 & 3065 & 3356 \\
    8 & 0 & 0 & 0 & 0 & 0 & 0 & 0 & 1595 & 3682 & 4120 & 4542 & 4980 \\
    9 & 0 & 0 & 0 & 0 & 0 & 0 & 0 & 0 & 2527 & 5775 & 6380 & 6994\\
    10 & 0 & 0 & 0 & 0 & 0 & 0 & 0 & 0 & 0 & 3851 & 8687 & 9534\\
    11 & 0 & 0 & 0 & 0 & 0 & 0 & 0 & 0 & 0 & 0 & 5610 & 12555\\
    12 & 0 & 0 & 0 & 0 & 0 & 0 & 0 & 0 & 0 & 0 & 0 & 7949\\
    \end{tabular}
    \caption{The number of four-point sets with multi-width \((w_1,w_2)\) up to affine equivalence.}
    \label{tab:quad}
\end{table}
A \emph{quasi-polynomial} is a polynomial whose coefficients are periodic functions with integral period.
By Theorem~\ref{thm:main} and \cite[Theorem 1.2]{triangles} the functions counting lattice triangles and width 1 lattice tetrahedra are piecewise quasi-polynomials whose coefficients have period 2 so we may expect a function counting four-point sets in the plane to be similar.
By Proposition~\ref{prop:four_pt_set_bound}, if there is a quasi-polynomial counting four-point sets of multi-width \((w_1,w_2)\) we expect it to be at most quadratic in \(w_2\).
We expect the case when \(w_1=w_2\) to be distinct due to the increased symmetry.
Also we expect the cases when \(w_2\) is odd and even to be distinct so consider them separately.
By fitting a quadratic to the results for \((w_1,w_1+1),\dots, (w_1,w_1+5)\) and \((w_1,w_1+2),\dots,(w_1,w_1+6)\) we obtain the following conjecture which agrees with the entries of Table~\ref{tab:quad}.
\begin{conjecture}
The number of four-point sets of multi-width \((w_1,w_2)\) if \(w_1<w_2\) is
\[
\begin{cases}
9w_2 + 6 & \text{if \(w_2\) even}\\
9w_2 + 4 & \text{if \(w_2\) odd}
\end{cases}
\]
when \(w_1=2\),
\[
\begin{cases}
\frac{47}{2}w_2 + 7 & \text{if \(w_2\) even}\\
\frac{47}{2}w_2 + \frac{11}{2} & \text{if \(w_2\) odd}
\end{cases}
\]
when \(w_1=3\),
\[
\begin{cases}
56w_2 + 6 & \text{if \(w_2\) even}\\
56w_2 + 2 & \text{if \(w_2\) odd}
\end{cases}
\]
when \(w_1=4\),
\[
\begin{cases}
\frac{211}{2}w_2 - 9 & \text{if \(w_2\) even}\\
\frac{211}{2}w_2 - \frac{23}{2} & \text{if \(w_2\) odd}
\end{cases}
\]
when \(w_1=5\) 
\[
\begin{cases}
183w_2 - 36 & \text{if \(w_2\) even}\\
183w_2 - 42 & \text{if \(w_2\) odd}
\end{cases}
\]
when \(w_1=6\)
\[
\begin{cases}
\frac{575}{2}w_2-94 & \text{if \(w_2\) even}\\
\frac{575}{2}w_2-\frac{195}{2} & \text{if \(w_2\) odd}
\end{cases}
\]
when \(w_1=7\) and
\[
\begin{cases}
430w_2-180 & \text{if \(w_2\) even}\\
430w_2-188 & \text{if \(w_2\) odd}
\end{cases}
\]
when \(w_1=8\).
\end{conjecture}
It is tempting to fit cubics in \(w_1\) to the coefficients of these polynomials to get a quasi-polynomial counting four-point sets whenever \(w_1 < w_2\).
However, the result is
\[
\begin{cases}
(\frac56w_1^3 + \frac16 w_1 + 2)w_2 -\frac54w_1^3+\frac{39}{4}w_1^2 - \frac{47}{2}w_1 + 24 & \text{ if \(w_1,w_2\) even}\\
(\frac56w_1^3 + \frac16 w_1 + 2)w_2 -\frac54w_1^3 + \frac{39}{4}w_1^2 - \frac{49}{2}w_1 + 24 & \text{ if \(w_1\) even and \(w_2\) odd}\\
(\frac56w_1^3 + \frac16w_1 + \frac12)w_2 -w_1^3 + \frac{51}{8}w_1^2 - 10w_1 + \frac{53}{8} & \text{ if \(w_1\) odd and \(w_2\) even}\\
(\frac56w_1^3 + \frac16w_1 + \frac12)w_2 -w_1^3 + \frac{51}{8}w_1^2 - \frac{21}{2}w_1 + \frac{53}{8} & \text{ if \(w_1,w_2\) odd}
\end{cases}
\]
which disagrees with Table~\ref{tab:quad} whenever \(w_1 \geq 9\).
This suggests that either there is no such quasi-polynomial or that for small values of \(w_1\) we have special cases and so cannot predict it from this data.
Taking successive differences of a sequence can help to identify when it is given by a quasi-polynomial since higher order terms cancel making the pattern more obvious.
Considering successive differences (and successive differences of these differences etc.) of the sequence counting multi-width \((w_1,w_1+1)\) four-point sets, it seems that if such a quasi-polynomial exists we would need significantly more data-points to estimate it.
Therefore, we do not attempt to classify enough four-point sets to make such a conjecture.

Using the classification of four-point sets, we move on to classify tetrahedra.
This uses a similar algorithm to the four-point set case (see Algorithm~\ref{alg:tet}) with two main differences.
We may no longer assume that all the tetrahedra we want to classify are contained in a \(w_1 \times w_2 \times w_3\) box so must allow more \(z\)-coordinates to be assigned to each point.
Also, since we are not extending this classification to a higher dimension, we need only store the normal form of each tetrahedron in order to count them.

\begin{algorithm}
\DontPrintSemicolon
\KwData{The set \(\cA\) containing all four-point sets in the plane with multi-width \((w_1,w_2)\) written as a subset of \([0,w_1] \times [0,w_2]\).}
\KwResult{The set \(\cT\) containing all tetrahedra with multi-width \((w_1,w_2,w_3)\).}
\(\cT \longleftarrow \emptyset\)\;
\For{\(\{v_1,v_2,v_3,v_4\} \in \cA\)}{
    \For{\(h_1,h_2,h_3,h_4 \in [0,\max\{w_1+w_2,w_3\}]\cap \ZZ\) such that \(h_i=0\) and \(h_j  \geq w_3\) for some \(i<j\)}{
        \(T \longleftarrow \conv(v_i\times\{h_i\} : i=1,2,3,4)\)\;
        \If{\(\mwidth(T) = (w_1,w_2,w_3)\)}{
            \(\cT \longleftarrow \cT \cup \{\NF(T)\}\)\;
        }
    }
}
\caption{Classifying the tetrahedra with multi-width \((w_1,w_2,w_3)\).}
\label{alg:tet}
\end{algorithm}

Based on Theorem~\ref{thm:main} we may hope that the tetrahedra of multi-width \((2,w_2,w_3)\) are counted by some quadratic functions in \(w_2\).
Since we can fit a quadratic to any three points we would like at least 4 points in each subsequence of \(|\cT_{2,w_2,w_3}|\) to make a reasonable conjecture.
Including the case \(w_2=2\) makes the resulting polynomials higher degree so we need to classify at least multi-width \((2,w_2,w_2)\), \((2,w_2,w_2+1)\) and \((2,w_2,w_2+2)\) tetrahedra for \(w_2 = 3,\dots,10\) to get enough data points.
The classification of lattice tetrahedron with width 2, second width up to 10 and third width up to 12 can be found in the database \cite{data} and they are counted in Table~\ref{tab:width_2_tet}.
Fitting polynomials to the sequences displayed in Table~\ref{tab:width_2_tet} produces the following.
\begin{conjecture}\label{conj:width_2_tet}
\begin{itemize}
The cardinality of \(\cT_{2,w_2,w_3}\) is given by the following:
    \item When \(w_3>w_2>2\) the cardinality of \(\cT_{2,w_2,w_3}\) is
    \begin{itemize}
        \item \(\frac14(81w_2^2-18w_2+76)\) if \(w_2\) and \(w_3\) even
        \item \(\frac14(81w_2^2-18w_2+56)\) if \(w_2\) even and \(w_3\) odd
        \item \(\frac14(81w_2^2-18w_2+37)\) if \(w_2\) odd and \(w_3\) even
        \item \(\frac14(81w_2^2-18w_2+25)\) if \(w_2\) and \(w_3\) odd
    \end{itemize}
    \item when \(w_2>2\) the cardinality of \(\cT_{2,w_2,w_2}\) is
    \begin{itemize}
        \item \(\frac18(81w_2^2-22w_2+80)\) if \(w_2\) is even
        \item \(\frac18(81w_2^2-20 w_2+27)\) if \(w_2\) is odd
    \end{itemize}
    \item when \(w_3>2\) the cardinality of \(\cT_{2,2,w_3}\) is
    \begin{itemize}
        \item \(47\) if \(w_3\) is even
        \item \(45\) if \(w_3\) is odd
    \end{itemize}
    \item and  the cardinality of \(\cT_{2,2,2}\) is \(17\).
\end{itemize}  
\end{conjecture}
More generally we may also guess that the following pattern will continue to hold.
\begin{conjecture}
    There is a piecewise quasi-polynomial with 4 components counting lattice tetrahedra of multi-width \((w_1,w_2,w_3)\).
    There is a component for each combination of equalities in \(w_3 \geq w_2 \geq w_1 > 0\).
    The leading coefficient in the case \(w_3 > w_2 > w_1 > 0\) is double the leading coefficient in the case \(w_3 = w_2 > w_1 > 0\).
    For fixed \(w_1\) and \(w_2\) there are at most three values which \(|\cT_{w_1,w_2,w_3}|\) can take depending of whether \(w_3\) is odd, even or equal to \(w_2\).
\end{conjecture}

\bibliographystyle{alpha}
\bibliography{TetBib}{}
\end{document}